\documentclass[a4paper,twoside,10pt,reqno]{article}
\usepackage[english]{babel}
\usepackage[dvips]{graphicx}
\usepackage[T1]{fontenc}
\usepackage[latin1]{inputenc}
\usepackage{amssymb,amsmath,amsthm,dsfont,mathrsfs}
\usepackage{amsfonts,euscript}
\usepackage{color}
\usepackage{authblk}

\newtheorem{teo}{Theorem}[section]
\newtheorem{pro}[teo]{Proposition}
\newtheorem{cor}[teo]{Corollary}

\newtheorem{ex}{Example}[section]

\newcommand{\esp}{ }

\newcommand{\x}{$\{X_n\}_{n\geq 1}$}
\newcommand{\z}{$\{Z_n\}_{n\geq 1}$}
\newcommand{\w}{$\{W_n\}_{n\geq 1}$}
\newcommand{\s}{$\{S_n\}_{n\geq 1}$}

\newcommand{\dst}{\displaystyle}

\newcommand{\nn}{\nonumber}
\newcommand{\nid}{\noindent }

\newcommand{\titulo}[1]{\begin{center}\mbox{} \\ \noindent \textit{\textbf{\Large #1}}\\\vspace{0.5cm}\end{center}}

\renewcommand{\abstract}[1]{{\small \noindent \textbf{Abstract:} #1\\}}

\begin{document}
\pagestyle{myheadings}

\begin{center}
\titulo{Estimating the extremal index through local dependence}
\end{center}

\vspace{0.5cm}

\textbf{Helena Ferreira} Department of Mathematics, University of
Beira
Interior, Covilhã, Portugal  (helena.ferreira@ubi.pt)\\

\textbf{Marta Ferreira} Center of Mathematics of Minho University, Braga, Portugal (msferreira@math.uminho.pt)\\

\abstract{The extremal index is an important parameter in the characterization of extreme values of a stationary sequence. Our new estimation approach for this parameter is based on the extremal behavior under the local dependence condition D$^{(k)}$($u_n$). We compare a process satisfying one of this hierarchy of increasingly weaker local mixing conditions with a process of cycles satisfying the D$^{(2)}$($u_n$) condition. We also analyze local dependence within moving maxima processes and derive a necessary and sufficient condition for D$^{(k)}$($u_n$). In order to evaluate the performance of the proposed estimators, we apply an empirical diagnostic for local dependence conditions, we conduct a simulation study and compare with existing methods. An application to a financial time series is also presented.}

\nid\textbf{keywords:} {extreme value theory, stationary sequences, dependence conditions, extremal index}\\

\nid\textbf{AMS 2000 Subject Classification}  Primary: 60G70; Secundary: 62G32\\

\section{Introduction}\label{sint}

Let $\{X_n\}_{n\geq 1}$ be a stationary sequence with marginal distribution $F_{X_n}=F$. Consider $M_{i,j}=\bigvee_{s=i+1}^jX_s$, where $x\vee y$ denotes $\max(x,y)$, with $M_{0,j}=M_j$ and $M_{i,j}=-\infty$ for $i>j$. The sequence \x\esp has extremal index $\theta\in[0,1]$ if, for each $\tau>0$, there is a sequence of normalized levels  $\{u_n\equiv u_n^{(\tau)}\}_{n\geq 1}$, i.e.,
\begin{eqnarray}\nn
n(1-F(u_n))\to\tau,
\end{eqnarray}
as $n\to\infty$, such that
\begin{eqnarray}\label{defei}
P(M_n\leq u_n)\to e^{-\theta\tau},
\end{eqnarray}
(Leadbetter \emph{et al.} \cite{lead+83} 1983).
When $\theta=1$ the exceedances of high thresholds $u_n$, by the variables in \x, tend to occur isolated as in an independent variables context. However, if $\theta<1$ we have groups of exceedances in the limit.
Clusters of extreme values are linked with incidences and durations of catastrophic phenomena, an important issue in areas like environment, finance, insurance, engineering, among others. The extremal index is a key parameter in this context and its estimation has been greatly addressed in literature. The most popular procedures are the blocks and the runs estimators (e.g., Nandagopalan \cite{nand90} 1990; Hsing \cite{hsing93} 1993; Weissman and Novak \cite{weis+98} 1998; Robert \emph{et al.} \cite{rob+09} 2009), and more recently, the interexceedance times method (Ferro and Segers \cite{ferro+segers03} 2003; S\"{u}veges \cite{suv07}).
The first group estimators requires a clustering identification parameter which is a largely arbitrarily task to comply and has some impact in inference. The second group avoids this parameter but may be less precise in specific contexts.

In this paper we propose a new estimation procedure that works under the local dependence condition D$^{(k)}$($u_n$) of Chernick \emph{et al.} (\cite{chern+91}, 1991). This condition requires the dependence condition D($u_n$) of Leadbetter (\cite{lead74}, 1974), which states that $\alpha_{n,l_n}\to 0$, as $n\to\infty$, for some sequence $l_n=o(n)$, where
\begin{eqnarray}\nn
\begin{array}{l}
\alpha_{n,l}=\sup\{|P\left(M_{i_1,i_1+p}\leq u_n,M_{j_1,j_1+q}\leq u_n\right)
-P\left(M_{i_1,i_1+p}\leq u_n\right)P\left(M_{j_1,j_1+q}\leq u_n\right)|:\vspace{0.35cm}\\
\hspace{1cm}\,1\leq i_1<i_1+p+l\leq j_1<j_1+q\leq n\}.
\end{array}
\end{eqnarray}
We say that condition D$^{(k)}$($u_n$) holds for \x, if for some $\{k_n\}_{n\geq 1}$ such that,
\begin{eqnarray}\label{kn}
k_n\to\infty,\,k_n\alpha_{n,l_n}\to 0,\,k_nl_n/n\to 0,
\end{eqnarray}
as $n\to\infty$, we have
\begin{eqnarray}\nn
nP\left(X_1>u_n,M_{1,k}\leq u_n<M_{k,r_n}\right)\dst\mathop{\longrightarrow}_{n\to\infty} 0,
\end{eqnarray}
with $\{r_n=[n/k_n]\}_{n\geq 1}$ ($[x]$ denotes the integer part of x).
Condition D$^{(k)}$($u_n$) is implied by
\begin{eqnarray}\nn
n\sum_{j=k+1}^{r_n}P\left(X_1>u_n,M_{1,k}\leq u_n<X_{j}\right)\dst\mathop{\longrightarrow}_{n\to\infty} 0.
\end{eqnarray}
This corresponds to condition D$^{'}$($u_n$) of Leadbetter \emph{et al.} (\cite{lead+83}, 1983) whenever $k=1$ which locally restricts the occurrence of clusters of exceedances and thus leads to $\theta=1$. If $k=2$ we have condition D$^{''}$($u_n$) of Leadbetter and Nandagopalan (\cite{lead+nand89}, 1989). This condition locally restricts the occurrence of two or more upcrossings, but still allows clustering of exceedances.


In Chernick \emph{et al.} (\cite{chern+91}, 1991) it is proved that, under D$^{(k)}$($u_n$), the extremal index exists and is given by
\begin{eqnarray}\label{eiruns}
\theta_X=\lim_{n\to\infty}P(M_{1,k}\leq u_n|X_1>u_n).
\end{eqnarray}
The runs estimator can be derived from this relation by taking the runs parameter $r$ as $k$. In particular, under condition D$^{(2)}$($u_n$), Nadagopalan (\cite{nand90}, 1990) found
$$
\begin{array}{rl}
\theta_X=& \dst\lim_{n\to\infty}P(X_2\leq u_n|X_1>u_n)\vspace{0.35cm}\\
=& \dst\lim_{n\to\infty}\frac{P(X_1\leq u_n<X_2)}{P(X_2>u_n)},
\end{array}
$$
which motivates his estimator based on the ratio between the number of upcrossings (equal to the number of downcrossings) and the number of exceedances. Although the D$^{(2)}$($u_n$) condition implies D$^{(k)}$($u_n$) for $k>2$ and we have several representations for $\theta_X$ as in (\ref{eiruns}), under D$^{(2)}$($u_n$) we have only to be concerned with the count of upcrossings and exceedances, rather than the length $r$ for runs of non-exceedances or intervals between exceedances. It is this easy approach in the Nandagopalan's estimator that we want to take advantage in this paper, by estimating $\theta_X$ through the extremal index of an auxiliary sequence satisfying D$^{(2)}$($u_n$).

The results that motivate our new estimation approach are given in Section \ref{sresult}. In this section we relate the extremal index $\theta_X$ of the process satisfying D$^{(k)}$($u_n$) with the one of a process of cycles satisfying D$^{(2)}$($u_n$), deriving new representations for $\theta$ that motivate the estimators. In this way, we promote the application of the estimation procedures that work under D$^{(2)}$($u_n$).

Knowledge about D$^{(k)}$($u_n$) has not only impact on the computation of the extremal index of a process but also informs about the cluster structure of extreme values. In moving maximum processes we directly obtain the extremal index by calculating the limit in (\ref{defei}). It may be the reason why there is no study in literature, as far as we know, concerning local dependence within these processes.
In Section \ref{smm} we derive a necessary and sufficient condition for D$^{(k)}$($u_n$) to hold within moving maxima processes.

Section \ref{sestim} is devoted to inference. We state a diagnostic tool to analyze D$^{(k)}$($u_n$) since it is the context of our framework. Therefore, we are also moving forward in diminishing the arbitrarity in the declustering scheme of the runs estimator. We analyze the performance of the new estimators trough simulation and illustrate with an application to a financial time series. We conclude in Section \ref{sdisc}.

\section{Extremal index of grouped variables}\label{sresult}

Let $\{I_0=0,I_n,n\geq 1\}$ be an increasing sequence of integer random variables (r.v.'s) such that $\{S_n=I_n-I_{n-1}\}_{n\geq1}$ is an i.i.d.~sequence satisfying $E(S_n)=p$. From such a renewal process and a stationary sequence \x, let define
\begin{eqnarray}\label{z}
Z_n=M_{I_{n-1},I_n},\,n\geq 1.
\end{eqnarray}
Driven by the strategy used by Rootz\'{e}n (1988, \cite{root88}) in the study of the extremal behavior of the regenerative processes, we will compare $M_n$ with the maximum of the first $[n/p]$ variables in the sequence of cycles \z.
\begin{pro}\label{p1}
Let \x \esp be a stationary sequence and \z\esp defined by (\ref{z}), for some renewal process \s\esp such that $E(S_n)=p$. If $\left\{n\bigvee_{i\geq 1}P\left(M_{I_{i-1},I_i}>u_n\right)\right\}_{n\geq 1}$ is bounded, then
\begin{eqnarray}\nn
P(M_n\leq u_n)-P\left(\bigcap_{i=1}^{[n/p]}M_{I_{i-1},I_i}\leq u_n\right)\to 0,\,n\to\infty.
\end{eqnarray}
\end{pro}
\begin{proof}
Let $L_n=\sup\{k:I_k\leq n\}$ and $U_n=\inf\{k:I_k> n\}$. By the law of large numbers, $\forall\epsilon>0$, we have
$$
P\left(\left|\frac{L_n}{n}-\frac{1}{p}\right|>\epsilon\right)\dst\mathop{\to }_{n\to\infty} 0\,\textrm{ and }\,
P\left(\left|\frac{U_n}{n}-\frac{1}{p}\right|>\epsilon\right)\dst\mathop{\to }_{n\to\infty} 0.
$$
Furthermore,
$$
\begin{array}{rl}
P(M_n\leq u_n)&\leq P\left(\bigcap_{i=1}^{L_n}\{M_{I_{i-1},I_i}\leq u_n\}\right)\vspace{0.35cm}\\
&=P\left(\bigcap_{i=1}^{L_n}\{M_{I_{i-1},I_i}\leq u_n\},n(1/p-\epsilon)\leq L_n\leq n(1/p+\epsilon)\right)+o(1)\vspace{0.35cm}\\
&\leq P\left(\bigcap_{i=1}^{[n(1/p-\epsilon)]}\{M_{I_{i-1},I_i}\leq u_n\}\right)+o(1)
\end{array}
$$
Similarly, we derive
$$
\begin{array}{rl}
P(M_n\leq u_n)&\geq P\left(\bigcap_{i=1}^{[n(1/p+\epsilon)]}\{M_{I_{i-1},I_i}\leq u_n\}\right)+o(1).
\end{array}
$$
Now, just observe that
$$
\begin{array}{l}
0\leq P\left(\bigcap_{i=1}^{[n/p]}\{M_{I_{i-1},I_i}\leq u_n\}\right)-P\left(\bigcap_{i=1}^{[n(1/p+\epsilon)]}\{M_{I_{i-1},I_i}\leq u_n\}\right)\vspace{0.35cm}\\
\leq n\epsilon \bigvee_{i\geq 1}\{P\left(M_{I_{i-1},I_i}> u_n\right)\}
\leq \epsilon k,
\end{array}
$$
as well as
$$
\begin{array}{l}
0\leq P\left(\bigcap_{i=1}^{[n(1/p-\epsilon)]}\{M_{I_{i-1},I_i}\leq u_n\}\right)-P\left(\bigcap_{i=1}^{[n/p]}\{M_{I_{i-1},I_i}\leq u_n\}\right)\leq \epsilon k,
\end{array}
$$
for some constant $k$.
\end{proof}

If we assume that \z\esp is stationary satisfying D($u_n$) and a local dependence condition D$^{(k)}$($u_n$) then, by applying the previous proposition, we can compute the extremal index of \x\esp from the knowledge of the joint distribution of a finite number of consecutive terms of \z.

In what concerns the local behavior of the large values of \z, we are going to consider to ways: in Proposition \ref{peiindep} we derive the extremal index by assuming the local independence condition D$^{(1)}$($u_n$) and in Proposition \ref{peidep2} by assuming the local dependence condition D$^{(2)}$($u_n$).

\begin{pro}\label{peiindep}
Under the conditions of Proposition \ref{p1}, if  \z\esp is stationary and satisfies D($u_n$) and D$^{(1)}$($u_n$) conditions for $u_n$ such that $un\equiv u_n^{(\tau)}$ for \x\esp and $u_n\equiv u_n^{(\tau^*)}$ for \z, then \x\esp has extremal index
$$
\theta_X=\lim_{n\to\infty}\frac{P(M_{I_1}>u_n)}{pP(X_1>u_n)}=\frac{\tau^*}{p\tau}.
$$
\end{pro}
\begin{proof}
We have that \z\esp has extremal index $\theta_Z=1$ and thus, by applying Proposition \ref{p1},
$$
\lim_{n\to\infty}P(M_n\leq u_n)=\lim_{n\to\infty}P\left(\bigvee_{i=1}^{[n/p]}Z_i\leq u_n\right)
=\lim_{n\to\infty}(P(M_{I_1}\leq u_n))^{[n/p]}=e^{-\tau^*/p}.
$$
Therefore, $\lim P(M_n\leq u_n)=e^{-\theta_X\tau}$, as $n\to\infty$, with
$$
\theta_X=\frac{\tau^*}{\tau p}=\lim_{n\to\infty}\frac{P(M_{I_1}>u_n)}{P(X_1>u_n)p}.
$$
\end{proof}

This is what happens in regenerative processes with independent cycles (see expression (4.2) in Rootz\'{e}n, \cite {root88} 1988, obtained directly).

\begin{pro}\label{peidep2}
Under the conditions of Proposition \ref{p1}, if \z\esp is stationary and satisfies D($u_n$) and D$^{(2)}$($u_n$) conditions for $u_n$ such that $un\equiv u_n^{(\tau)}$ for \x, $u_n\equiv u_n^{(\tau^*)}$ for \z\esp and $nP\left(M_{I_1}\leq u_n< M_{I_1,I_2}\right)\to \nu^*$, then \x\esp has extremal index
$$
\theta_X=\frac{\theta_Z\tau^*}{p\tau}=\lim_{n\to\infty}\frac{P(M_{I_1}\leq u_n<M_{I_{1},I_2})}{pP(X_1>u_n)}=\frac{\nu^*}{p\tau}.
$$
\end{pro}
\begin{proof}
We have that \z\esp has extremal index
\begin{eqnarray}\label{peidep2.ei1}
\theta_Z=\lim_{n\to\infty}\frac{P(M_{I_1}\leq u_n<M_{I_{1},I_2})}{P(M_{I_1}>u_n)}=\frac{\nu^*}{\tau^*}
\end{eqnarray}
and thus, by applying Proposition \ref{p1},
$$
\lim_{n\to\infty}P(M_n\leq u_n)=\lim_{n\to\infty}P\left(\bigvee_{i=1}^{[n/p]}Z_i\leq u_n\right)
=e^{-\theta_Z\tau^*/p}=e^{-\tau\theta_X},
$$
with
\begin{eqnarray}\label{peidep2.ei3}
\theta_X=\frac{\theta_Z\tau^*}{p\tau}=\lim_{n\to\infty}\frac{P(M_{I_1}\leq u_n<M_{I_{1},I_2})}{pP(X_1>u_n)}.
\end{eqnarray}
\end{proof}
This is what happens in regenerative processes with $1-$dependent cycles.
(see comment after expression (4.2) in Rootz\'{e}n, \cite {root88} 1988).\\

Since
$$
nP\left(Z_1>u_n,\bigvee_{i=2}^{r_n}Z_i>u_n\right)=nP\left(Z_1>u_n,Z_2\leq u_n<\bigvee_{i=3}^{r_n}Z_i\right)+nP\left(Z_1>u_n,Z_2>u_n\right)
$$
and
$$
nP\left(Z_1>u_n,Z_2>u_n\right)=nP\left(Z_1>u_n\right)-
nP\left(Z_1>u_n\geq Z_2\right),
$$
we can remark that, for \z\esp satisfying D$^{(2)}$($u_n$), it holds that \z\esp satisfies D$^{(1)}$($u_n$) if and only if $\tau^*=\nu^*$, that is, the limiting mean number of exceedances is asymptotically equal to the limiting mean number of upcrossings (or downcrossings). Also, for any $k>2$, provided that \x\esp satisfies  D$^{(k)}$($u_n$), it holds D$^{(k-1)}$($u_n$) if and only if
$$
\lim_{n\to\infty}n\left(P\left(X_1>u_n,M_{1,k-1}\leq u_n\right)-P\left(X_1>u_n,M_{1,k}\leq u_n\right)\right)=0.
$$
This remark will help in the choice of a value for $k$, in Section \ref{sestim}, dedicated to the estimation of $\theta_X$.\\

The presented results also point out a way to obtain the limiting law of the maximum term of the first $\sum_{i=1}^nS_i$ r.v.'s of the sequence \x. In fact, for \x\esp and \s\esp as above, it holds that
$$
\begin{array}{rl}
&\dst P\left(M_{\sum_{i=1}^nS_i}\leq u_n\right)=\left(\bigvee_{j=1}^{\sum_{i=1}^nS_i}X_j\leq u_n\right)=P(Z_1\leq u_n,\hdots,Z_n\leq u_n)\vspace{0.35cm}\\
\dst\mathop{\to}_{n\to\infty}&\dst e^{-\theta_Z\tau^*}=e^{-p\tau\theta_X}
=\exp\left\{-E(S_1)\lim_{n\to\infty}nP(X_1>u_n,X_2\leq u_n,\hdots,X_k\leq u_n)\right\},
\end{array}
$$
if \x\esp satisfies D$^{(k)}$($u_n$).\\

We could state a general result analogous to the above Propositions \ref{peiindep} and \ref{peidep2}, by considering \z\esp satisfying D$^{(k)}$($u_n$) with $k>2$. However, our final goal is to relate the extremal index of a sequence \x\esp satisfying D$^{(k)}$($u_n$) with $k>2$, with the extremal index of an auxiliary sequence \z\esp satisfying D$^{(k)}$($u_n$) with $k\leq 2$. This will enable to take profit of the estimation of the extremal index under D$^{(1)}$($u_n$) or D$^{(2)}$($u_n$), after a suitable transformation of the data. The identification of clusters reduces then to the identification of blocks of consecutive exceedances. Next results discuss relations on long-range and local dependence conditions for \x\esp and \z, which can be easily obtained in the particular case of a deterministic $I_n$, $n\geq 0$, considered later for the main proposal of this work.

\begin{pro}\label{pcondD}
Let \x \esp be such that $\{X_i,\,i\in B\}$ is independent of $\{S_i,\,i\in A\}$ whenever $A\cap B\not=0$, for some renewal process $S=\{I_0=0,S_n=I_n-I_{n-1}\}_{n\geq 1}$ with $t\leq S_n\leq s$, $n\geq 1$. If, conditionally on $S$, \x\esp satisfies condition D($u_n$) with spacer sequence $l_n$, then \z,  defined by (\ref{z}), satisfies D($u_n$) with $l_n^*=[2l_n/t]$.
\end{pro}
\begin{proof}
Let $I=\{i_1,\hdots,i_p\}$ and $J=\{j_1,\hdots,j_q\}$ with $1\leq i_1<\hdots<i_p<i_p+l_n^*<j_1<\hdots<j_q\leq n$ and $l_n^*=[2l_n/t]$. Then
\begin{eqnarray}\nn
\begin{array}{rl}
&\left|P\left(\bigvee_{i\in I}Z_i\leq u_n,\bigvee_{i\in J}Z_i\leq u_n\right)-P\left(\bigvee_{i\in I}Z_i\leq u_n\right)P\left(\bigvee_{i\in J}Z_i\leq u_n\right)\right|\vspace{0.35cm}\\
=& \left|E\left(P\left(\bigvee_{i\in I^*(S_I)}X_i\leq u_n,\bigvee_{i\in J^*(S_J)}X_i\leq u_n\right)|S\right)\right.\vspace{0.35cm}\\
&\left.-E\left(P\left(\bigvee_{i\in I^*(S_I)}X_i\leq u_n\right)|S\right)E\left(P\left(\bigvee_{i\in J^*(S_J)}X_i\leq u_n\right)|S\right)\right|,
\end{array}
\end{eqnarray}
where $S_A$ denotes the vector of r.v.'s $S_i$, $i\in A$, and $I^*(S_I)$ and $J^*(S_J)$ are separated by at least $l_n$, since we have $I_{i_p}\leq i_ps$, $I_{j_1-1}\geq (j_1-1)t$ and thus $(j_1-1)t+1-i_ps\geq l_n$, for large enough $n$.
Therefore, and meeting the given assumptions, the previous expression is upper bounded by
\begin{eqnarray}\nn
\begin{array}{rl}
& \left|E\left(P\left(\bigvee_{i\in I^*(S_I)}X_i\leq u_n,\bigvee_{i\in J^*(S_J)}X_i\leq u_n\right)|S_{I\cup J}\right)\right.\vspace{0.35cm}\\
&\left.-E\left(P\left(\bigvee_{i\in I^*(S_I)}X_i\leq u_n\right)P\left(\bigvee_{i\in J^*(S_J)}X_i\leq u_n\right)|S_{I\cup J}\right)\right|
\vspace{0.35cm}\\
+&\left|E\left(P\left(\bigvee_{i\in I^*(S_I)}X_i\leq u_n\right)P\left(\bigvee_{i\in J^*(S_J)}X_i\leq u_n\right)|S_{I\cup J}\right)\right.
\vspace{0.35cm}\\
&\left.-E\left(P\left(\bigvee_{i\in I^*(S_I)}X_i\leq u_n\right)|S_{I}\right)E\left(P\left(\bigvee_{i\in J^*(S_J)}X_i\leq u_n\right)|S_{J}\right)\right|\vspace{0.35cm}\\
\leq &\alpha_{n,l_n}\,.
\end{array}
\end{eqnarray}
\end{proof}

\begin{pro}\label{peidep2.2}
Let \x \esp be a stationary sequence  and \z\esp  defined by (\ref{z}), for some renewal process $S=\{I_0=0,S_n=I_n-I_{n-1}\}_{n\geq 1}$ with $t\leq S_n\leq s$, $n\geq 1$.
\begin{itemize}
\item[(a)] If \z\esp satisfies D$^{(2)}$($u_n$) with $r_n=[n/k_n]$, then \x\esp satisfies D$^{(2s-t+1)}$($u_n$) with the same $r_n$.
\item[(b)] If $2t>s$ and \x\esp satisfies D$^{(k)}$($u_n$) for some $k\leq 2t-s+1$, with $r_n=[n/k_n]$, then \z\esp satisfies D$^{(2)}$($u_n$) with $r_n^*=[r_n/s]$.
\end{itemize}
\end{pro}
\begin{proof}
To obtain (a) we take into account the following inequalities:
\begin{eqnarray}\nn
\begin{array}{rl}
&nP\left(X_1>u_n,M_{1,2s-t+1}\leq u_n<M_{2s-t+1,r_n}\right)\vspace{0.35cm}\\
\leq & nP\left(X_1>u_n,M_{1,2s-t+1}\leq u_n<M_{2s-t+1,tr_n-t+1)}\right)\vspace{0.35cm}\\
= &nP\left(X_t>u_n,M_{t,2s}\leq u_n<M_{2s,tr_n}\right)\vspace{0.35cm}\\
\leq &nP\left(M_{I_1}>u_n,M_{I_1,I_2}\leq u_n<\bigvee _{i=3}^{r_n}M_{I_{i-1},I_i}\right)\vspace{0.35cm}\\
= &o(1),\,n\to\infty\,.
\end{array}
\end{eqnarray}
For (b) we have:
\begin{eqnarray}\nn
\begin{array}{rl}
&nP\left(M_{I_1}>u_n,M_{I_1,I_2}\leq u_n<\bigvee _{i=3}^{r_n^*}M_{I_{i-1},I_i}\right)\vspace{0.35cm}\\
\leq & nP\left(M_s>u_n,M_{s,2t}\leq u_n<\bigvee _{i=2t+1}^{r_n^*s}X_j\right)\vspace{0.35cm}\\
\leq &n\sum_{j=1}^sP\left(X_j>u_n,M_{j,s}\leq u_n,M_{s,2t}\leq u_n<M_{2t,r_n}\right)
\end{array}
\end{eqnarray}
and each of the $s$ terms in the sum above tends to zero by the D$^{(2t-s+1)}$($u_n$) condition for \x.
\end{proof}

We state now a result on the ``clustered" process \z\esp that resumes the path to obtain its extremal index $\theta_Z$ and to recover $\theta_X$ for the ``declustered" process \x, enhancing that to count the mean number of upcrossings (or downcrossings) for $\{Z_1,\hdots,Z_n\}$ is asymptotically equivalent to count the mean number of runs $\{X_i>u_n,X_{i+1}\leq u_n,\hdots,X_{i+k-1}\leq u_n\}$, $i\leq np$.

\begin{cor}
Let \x\esp and $S$ be in the conditions of Proposition \ref{pcondD} and that $E(S_n)=p$. Suppose that $u_n\equiv u_n^{(\tau)}$ for \x\esp and \x\esp satisfies D$^{(k)}$($u_n$) for some $k\leq 2t-s+1$. Then
\begin{itemize}
\item[(a)] \z\esp  defined by (\ref{z}) satisfies D($u_n$) and D$^{(2)}$($u_n$) conditions.
\item[(b)] If $u_n\equiv u_n^{(\tau^*)}$ for  \z\esp and there exists $\dst\mathop{\lim}_{n\to\infty}nP\left(M_{I_1}\leq u_n<M_{I_1,I_2}\right)=\nu^*$ then it holds that
    $$
    \theta_Z=\frac{\nu^*}{\tau^*},\,\theta_X=\frac{\theta_Z\tau^*}{p\tau}
    $$
     and
    $$
    \lim_{n\to\infty}nP\left(M_{I_1}\leq u_n<M_{I_1,I_2}\right)=p\lim_{n\to\infty}nP\left(X_1>u_n,X_2\leq u_n,\hdots,X_k\leq u_n\right).
    $$
\end{itemize}
\end{cor}
We now focus on the particular case of $I_n=n(k-1)$, $n\geq 0$, for some $k>2$. Therefore, we have $s=t=k-1$. If is this the case then, from the previous result, condition D$^{(2)}$($u_n$) for \z\esp implies condition D$^{(k)}$($u_n$) for \x\esp with the same length $r_n$ to model ``local behavior". Otherwise, the validity of condition D$^{(k)}$($u_n$) for \x\esp leads to condition D$^{(2)}$($u_n$) for \z\esp with $r_n^*=[r_n/(k-1)]$. We can then state the following corollary, which can also be proved directly.

\begin{cor}
Let \x\esp be a stationary sequence and \z\esp  defined by (\ref{z}) with $I_n=n(k-1)$, $n\geq 0$, for some $k>2$.Then \z\esp satisfies condition D$^{(2)}$($u_n$) if and only if \x\esp satisfies condition D$^{(k)}$($u_n$).
\end{cor}

Thus, under conditions of Proposition \ref{peidep2.2} and according to (\ref{peidep2.ei3}), the extremal index can be written as
\begin{flalign}
&\theta_X=\dst \frac{\theta_Z\tau^*}{(k-1)\tau}\vspace{0.35cm}\label{eidir}\\
=&\dst\lim_{n\to\infty}\frac{P(M_{I_1}\leq u_n<M_{I_{1},I_2})}{(k-1)P(X_1>u_n)}=\dst\lim_{n\to\infty}\frac{P(M_{k-1}\leq u_n<M_{k-1,2(k-1)})}{(k-1)P(X_1>u_n)}\nn.
\end{flalign}

By using the stationarity, this limit can be rewritten as the one obtained in Chernick \emph{et al.} (\cite{chern+91}, 1991) under condition D$^{(k)}$($u_n$) and given in (\ref{eiruns}).
However, representation (\ref{eidir}) allows to estimate $\theta_X$ through the estimation of $\theta_Z$ for the cycles $Z_i=M_{I_{i-1},I_i}$, $i\geq 1$, for which D$^{(2)}$($u_n$) holds, as we will present in Section \ref{sestim}.\\

Here we illustrate the above results with finite moving maxima processes and in the next section we devote special attention to the local dependence in this kind of processes.

\begin{ex}\label{ex1}
\emph{Consider the moving maximum process, $X_n = \bigvee_{j=0}^3\alpha_jY_{n-j}$, $\alpha_0=2/6$, $\alpha_1=1/6$ and $\alpha_2=3/6$, $n\geq 1$, with
sequence $\{Y_n\}_{n\geq -1}$ independent and having standard Fréchet marginal distribution, $F_{Y}=\exp(-1/x)$, $x>0$. This stationary sequence satisfies  D$^{(3)}$($u_n$) for levels $u_n=n/\tau$, $\tau>0$, as will be seen in the next section and $\theta_X=((2/6)\vee(1/6)\vee (3/6))=1/2$ (see Weissman and Cohen \cite{weis+cohen95}, 1995).  For $Z_n=X_{2n-1}\vee X_{2n}$, we have $nP(Z_1>u_n)\to (5/3)\tau=\tau^*$, as $n\to\infty$, since
$$
P(Z_1\leq u_n)=F_{Y}(u_n)F_{Y}(3u_n/2).
$$
Observe also that
$$
P(Z_2\leq u_n,Z_1\leq u_n)=F_{Y}(u_n/2)F_{Y}(3u_n/2)
$$
and, provided that D$^{(2)}$($u_n$) holds for \z,
we have
$$
\theta_Z=\lim_{n\to\infty}P(Z_2\leq u_n|Z_1>u_n)=3/5.
$$
By applying (\ref{eidir}), we obtain $\theta_X=1/2$.}\hfill $\square$\\
\end{ex}

We have seen that, for every stationary sequence \x\esp satisfying D$^{(k)}$($u_n$), $k>2$, we can build a stationary sequence \z\esp satisfying D$^{(2)}$($u_n$) by taking the maxima of $k-1$ consecutive variables of sequence \x. For big values of $k$, such aggregation can result in reduced accuracy in the estimation of $\theta_X$ via the sample based on \z, as will be pointed in Section \ref{sestim}. The Proposition \ref{peidep2.2} also states that it can be considered mixtures of big cycles of several lengths in order to build the sequence \z\esp satisfying D$^{(2)}$($u_n$), as it is illustrated in the next example.

\begin{ex}\label{ex1}
\emph{Let $S=\{I_0=0,S_n=I_n-I_{n-1}\}_{n\geq 1}$ be a sequence of independent r.v.'s uniformly distributed on $\{k,k+2\}$, with a fixed $k\geq 6$, and independent of \x. For \x\esp take a stationary sequence satisfying D$^{(5)}$($u_n^{(\tau)}$) and D($u_n^{(\tau)}$), for instance, a MM process with signatures $\alpha_{l,j}$ as given in the next section and $u_n^{(\tau)}=n/\tau$. Let $Z_n$ be as in (\ref{z}), which is stationary and also satisfies D($u_n$). Since $2k-(k+2)+1\geq 5$, then \z\esp satisfies condition D$^{(2)}$($u_n$). It holds that
$$
\tau^*=\lim_{n\to\infty}\frac{1}{2}\sum_{s\in\{k,k+2\}}
nP\left(\bigvee_{i=1}^sX_i>u_n\right),
$$
$$
\nu^*=\lim_{n\to\infty}\frac{1}{4}\sum_{s_1,s_2\in\{k,k+2\}}
nP\left(\bigvee_{i=1}^{s_1}X_i\leq u_n<\bigvee_{i=s_1+1}^{s_1+s_2}X_i\right),
$$
$\theta_Z=\frac{\nu^*}{\tau^*}$ and $\theta_X=\frac{\nu^*}{(k+1)\tau}$.
}\hfill $\square$\\
\end{ex}

\section{Condition D$^{(k)}$($u_n$) for max-stable processes}\label{smm}

A moving maxima process (MM) is defined as
\begin{eqnarray}\label{mm}
X_n=\bigvee_{l\geq 1}\bigvee_{-\infty<j<\infty}\alpha_{l,j}Y_{l,n-j},\,n\geq 1
\end{eqnarray}
where $\{Y_{l,j},\,l\geq 1,\,-\infty<j<\infty\}$ is an i.i.d.~sequence of r.v.'s, usually unit Fréchet and $\{\alpha_{l,j},\,l\geq 1,\,-\infty<j<\infty\}$ are non negative constants (usually denoted signatures) such that $\sum_{l\geq 1}\sum_{-\infty<j<\infty}\alpha_{l,j}=1$ (Deheuvels \cite{deheuv83} 1983, Davis and Resnick \cite{dav+res89} 1989, Smith and Weissman \cite{smi+weis96} 1996, Hall \emph{et al.} \cite{hall+02} 2002, Meinguet \cite{meinguet12} 2012).

Under the condition $\sum_{-\infty<j<\infty}\sum_{l\geq 1}l\alpha_{l,j}<\infty$, the MM process is strong-mixing (Meinguet, \cite{meinguet12} 2012) and therefore it satisfies D($u_n$).

An interesting feature of these processes is that the transformation of $\{Y_{l,j},\,l\geq 1,\,-\infty<j<\infty\}$  induces a dependence structure with extremes in temporal clusters. Any stationary process with finite distributions of multivariate extreme value type can be approximated by an MM process with marginals of extreme value type (Hall \emph{et al.} \cite{hall+02} 2002). Examples of finite MM processes (i.e., with $l$ and $j$ finite) are not difficult to deal with and are often used to illustrate long range and local dependence conditions within extreme values. The extremal index is directly obtained through $\lim_{n\to\infty}P(M_n\leq n/\tau)$, $\tau>0$, even for infinite MM and thus avoid the validity of some D$^{(k)}$ condition. In Meinguet (\cite{meinguet12} 2012) it was presented a nice finite-cluster condition which prevents a sequence of extremes occurring in MM from being infinite over time. However, it doesn't enable a representation for $\theta$ from finite marginal distributions of the process. The local dependence conditions brings us enlightenment about the clustering structure of extreme values. Any finite MM is $m$-dependent for some positive integer $m$ and thus D$^{(k)}$ holds at least for some $k\geq m$. From simple examples, we know that small changes in the values of coefficients $\alpha_{l,j}$ may lead to large diferences within the clusters structure. Hence, it is raised the question of which conditions must satisfy $\alpha_{l,j}$ so that some D$^{(k)}$ holds for an MM process. The next result presents a necessary and sufficient condition.

\begin{pro}\label{pmmdk}
Let \x\esp be an MM process as defined in (\ref{mm}), where $\{Y_{l,j},\,l\geq 1,\,-\infty<j<\infty\}$ is an i.i.d.~sequence of unit Fréchet r.v.'s. Then
\begin{itemize}
\item[(a)] \x\esp satisfies condition D$^{(k)}$($u_n$), $k\geq 2$, if and only if, for all $l\geq 1$ and $-\infty<j<\infty$,
\begin{eqnarray}\label{dkmm}
\alpha_{l,j}\wedge \left(\bigvee_{s\geq k+1}\alpha_{l,j+s-1}\right)\leq\bigvee_{s=2}^k\alpha_{l,j+s-1},
\end{eqnarray}
where $x\wedge y$ denotes $\min(x,y)$.
\item[(b)] \x\esp satisfies condition D$^{(1)}$($u_n$) if and only if, for all $l\geq 1$ there is only one $-\infty<j<\infty$ such that $\alpha_{l,j}>0$.
\end{itemize}
\end{pro}
\begin{proof}
\begin{itemize}
\item[(a)] Observe that
\begin{eqnarray}\label{pmmdk1}
\begin{array}{rl}
&\dst nP\left(X_1>u_n,\bigvee_{s=2}^{k}X_s\leq u_n,\bigvee_{s=k+1}^{r_n}X_s>u_n\right)\vspace{0.35cm}\\
=&\dst nP\left(X_1>u_n,\bigvee_{s=k+1}^{r_n}X_s>u_n\right)-nP\left(X_1>u_n,
\bigvee_{s=2}^{k}X_s>u_n,
\bigvee_{s=k+1}^{r_n}X_s>u_n\right).
\end{array}
\end{eqnarray}
Since $X_s=\bigvee_l\bigvee_j\alpha_{l,j}Y_{l,1-(j-s+1)}$ and $\{Y_{l,j},\,l\geq 1,\,-\infty<j<\infty\}$ is a sequence of independent r.v.'s, we have
\begin{eqnarray}\nn
\begin{array}{rl}
&\dst\lim_{n\to\infty}nP\left(X_1>u_n,\bigvee_{s=k+1}^{r_n}X_s>u_n\right)\vspace{0.35cm}\\
=&\dst\lim_{n\to\infty}n\sum_l\sum_jP\left(\alpha_{l,j}Y_{l,1-j}>u_n,\bigvee_{s=k+1}^{r_n}\alpha_{l,j+s-1}Y_{l,1-j}>u_n\right).
\end{array}
\end{eqnarray}
Thus, for levels $u_n=n/\tau$, $\tau>0$,
\begin{eqnarray}\nn
\begin{array}{rl}
&\dst\lim_{n\to\infty}nP\left(X_1>n/\tau,\bigvee_{s=k+1}^{r_n}X_s>n/\tau\right)\vspace{0.35cm}\\
=&\dst\lim_{n\to\infty}n\sum_l\sum_jP\left(Y_{l,1-j}>\frac{n/\tau}{\alpha_{l,j}}
\vee\frac{n/\tau}{\bigvee_{s=k+1}^{r_n}\alpha_{l,j+s-1}}\right)\vspace{0.35cm}\\
=&\dst\lim_{n\to\infty}n\sum_l\sum_jP\left(Y_{l,1-j}>\frac{n/\tau}{\alpha_{l,j}\wedge\bigvee_{s=k+1}^{r_n}\alpha_{l,j+s-1}}
\right).
\end{array}
\end{eqnarray}
Using the same reasoning on the second term in (\ref{pmmdk1}) and by the theorem of dominated convergence,
\begin{eqnarray}\nn
\begin{array}{rl}
&\dst \lim_{n\to\infty}nP\left(X_1>u_n,\bigvee_{s=2}^{k}X_s\leq u_n,\bigvee_{s=k+1}^{r_n}X_s>u_n\right)\vspace{0.35cm}\\
=&\dst \sum_l\sum_j\lim_{n\to\infty}n\left(F_Y\left(\frac{n/\tau}{\alpha_{l,j}\wedge\bigvee_{s=2}^{k}\alpha_{l,j+s-1}
\wedge\bigvee_{s=k+1}^{r_n}\alpha_{l,j+s-1}}\right)\right.\vspace{0.35cm}\\
& \dst\left.
-F_Y\left(\frac{n/\tau}{\alpha_{l,j}\wedge\bigvee_{s=k+1}^{r_n}\alpha_{l,j+s-1}}\right)\right)\vspace{0.35cm}\\
=&\dst \sum_l\sum_j\tau \left(\left(\alpha_{l,j}\wedge\bigvee_{s\geq k+1}\alpha_{l,j+s-1}\right)-\left(\alpha_{l,j}\wedge\bigvee_{s=2}^{k}\alpha_{l,j+s-1}
\wedge\bigvee_{s\geq k+1}\alpha_{l,j+s-1}\right)\right),
\end{array}
\end{eqnarray}
since $\alpha_{l,j}\to 0$, as $l\to\infty$ and $j\to\infty$.
Now just observe that the limit will be null whenever relation (\ref{dkmm}) holds, for all $l\geq 1$ and $-\infty<j<\infty$.
\item[(b)] The conclusion in (b) follows from
$$
\dst\lim_{n\to\infty}nP\left(X_1>u_n,\bigvee_{s=2}^{r_n}X_s>u_n\right)
=\sum_l\sum_j\tau\left(\alpha_{l,j}\wedge\bigvee_{s\geq 2}\alpha_{l,j+s-1}\right).
$$
The latter sum is null if and only if, for each $l$, there is $j^*$ such that  $\alpha_{l,j^*}>0$ and $\alpha_{l,j}=0$ for $j\not=j^*$.
\end{itemize}
\end{proof}

\begin{ex}\label{ex2}
\emph{Consider the moving maximum processes, $X_n = \bigvee_{j=0}^3\alpha_jY_{n-j}$, $\alpha_0=2/6$, $\alpha_1=1/6$ and $\alpha_2=3/6$, $n\geq 1$, given in Example \ref{ex1}, and $W_n = \bigvee_{j=0}^3\alpha_jY_{n-j}$, $\alpha_0=1/6$, $\alpha_1=3/6$ and $\alpha_2=2/6$, $n\geq 1$, with
sequence $\{Y_n\}_{n\geq -1}$ independent and having standard Fréchet marginal distribution. We will see that \x\esp satisfies  D$^{(3)}$($u_n$) (and not D$^{(2)}$($u_n$)) and that \w\esp satisfies  D$^{(2)}$($u_n$), for levels $u_n=n/\tau$, $\tau>0$, by applying relation (\ref{dkmm}). The calculations are summarized in Table \ref{tabex2}. Observe that D$^{(2)}$($u_n$) doesn't hold for \x\esp since, if $j=0$ then $\alpha_{0}\wedge \left(\bigvee_{s\geq 3}\alpha_{s-1}\right)=2/6\wedge 3/6> 1/6= \alpha_1$.}
\begin{table}
\caption{Verification of conditions D$^{(3)}$($u_n$) and D$^{(2)}$($u_n$) for, respectively, the MM processes \x\esp and \w\esp of Example \ref{ex2}, according to relation in (\ref{dkmm}). \label{tabex2}}
\begin{tabular}{ccc|ccc}
& condition D$^{(3)}$($u_n$)& for \x && condition D$^{(2)}$($u_n$)& for \w\\
\hline\hline
$j$ & $\alpha_{j}\wedge \left(\bigvee_{s\geq 4}\alpha_{j+s-1}\right)$ & $\alpha_{j+1}\vee \alpha_{j+2}$ & $j$ & $\alpha_{j}\wedge \left(\bigvee_{s\geq 3}\alpha_{j+s-1}\right)$ & $\alpha_{j+1}$\\
\hline
$\leq -3$& $0\wedge 3/6$ & $0$ & $\leq -2$& $0\wedge 3/6$ & $0$ \\
$-2,-1$& $0\wedge 3/6$ & $2/6$ & $-1$& $0\wedge 3/6$ & $1/6$\\
$0$& $2/6\wedge 0$ & $3/6$ & $0$& $1/6\wedge 2/6$ & $3/6$\\
$1$& $1/6\wedge 0$ & $3/6$ & $1$& $3/6\wedge 0$ & $2/6$\\
$2$& $3/6\wedge 0$ & $0$ & $2$& $2/6\wedge 0$ & $0$\\
$\geq 3$& $0\wedge 0$ & $0$ & $\geq 3$& $0\wedge 0$ & $0$\\
\end{tabular}
\end{table}
\hfill $\square$\\
\end{ex}

Inference within MM processes has been addressed in literature (see Zhang and Smith \cite{zhang+smith10}, 2010). Therefore, as an alternative to the empirical method of S\"{u}veges (\cite{suv07} 2007), we can check the validity of D$^{(k)}$($u_n$) within these processes by estimating coefficients $\alpha_{l,j}$ and applying (\ref{dkmm}).\\

The MM processes are stationary max-stable processes for which, under D$^{(k)}$($u_n$), we can derive the extremal index from a tail dependence coefficient.
Suppose that the stationary process \x\esp has unit Fréchet marginals $F(x)=\exp(-1/x)$, $x>0$. If D$^{(k)}$($u_n^{(\tau)}$) holds for \x, then
\begin{eqnarray}\label{eitdcm}
\begin{array}{rl}
\theta_X=&\dst\lim_{n\to\infty}P\left(M_{1,k}\leq n/\tau|X_1>n/\tau\right)\vspace{0.35cm}\\
=&1-\dst\lim_{n\to\infty}P\left(M_{1,k}> n/\tau|X_1>n/\tau\right)\vspace{0.35cm}\\
=&1-\Lambda_U^{(I_1|I_2)}(1,1),
\end{array}
\end{eqnarray}
provided the limit exists, where $I_1=\{2,\hdots,k\}$, $I_2=\{1\}$ and $\Lambda_U^{(I_1|I_2)}(1,1)$ is the upper tail dependence coefficient considered in Ferreira and Ferreira (\cite{fer+fer12b}, 2012b). In the case of max-stable processes or, more generally, processes satisfying the max-domain of attraction condition, the limit in (\ref{eitdcm}) is always defined. By applying the propositions 2.1 and 3.1 in Ferreira and Ferreira (\cite{fer+fer12b}, 2012b), we conclude that
\begin{eqnarray}\label{eiesp}
\theta_X=\frac{E\left(e^{-M_k^{-1}}\right)}{1-E\left(e^{-M_k^{-1}}\right)}
-\frac{E\left(e^{-M_{1,k}^{-1}}\right)}{1-E\left(e^{-M_{1,k}^{-1}}\right)}
\end{eqnarray}
and, in particular for $k=2$, it holds that
\begin{eqnarray}\label{eiespk2}
\theta_X=\frac{1}{1-E\left(e^{-(X_1\vee X_2)^{-1}}\right)}
-2
\end{eqnarray}
This representation for $\theta_X$ motivates its estimation from moment estimators for $E\left(e^{-M_k^{-1}}\right)=E\left(\bigvee_{i=1}^kF(X_i)\right)$, as considered in Ferreira and Ferreira (\cite{fer+fer12b}, 2012b).

We apply now the results of the previous section in order to compute $\theta_X$ from $\theta_Z$ of the process $\{Z_n=\bigvee_{i=(n-1)(k-1)+1}^{n(k-1)} X_i\}_{n\geq 1}$ under D$^{(2)}$($u_n$). The estimation of $\theta_Z$ is considerably more simpler as suggested by (\ref{eiespk2}).

\begin{pro}\label{pmsei}
Let \x\esp be a stationary max-stable process with unit Fréchet marginals $F$ and $u_n^{(\tau)}=n/\tau$, $\tau>0$. Then
\begin{itemize}
\item[(a)] \z\esp is stationary and max-stable with marginal distribution $F_Z(x)=F^{\epsilon_{k-1}}(x)$, where $\epsilon_{k-1}=-\log F_{(X_1,\hdots,X_{k-1})}(1,\hdots,1)\in [1,k-1]$ is the $(k-1)$-th extremal coefficient of \x.
\item[(b)] If \x\esp satisfies D($u_n$) and D$^{(k)}$($u_n$), $k> 2$, then \z\esp satisfies D($u_n$) and D$^{(2)}$($u_n$),
\begin{eqnarray}\label{eimsz}
\theta_Z=\frac{1}{1-E\left(F_{Z}(Z_1)\vee F_{Z}(Z_2)\right)}-2
\end{eqnarray}
and
\begin{eqnarray}\label{eimsx}
\theta_X=\theta_Z\frac{-\log F_Z(1)}{k-1}.
\end{eqnarray}
\end{itemize}
\end{pro}
\begin{proof}
We only justify (b), leaving (a) to the reader.

We first consider the sequence of cycles $\{Z_n^*=\bigvee_{i=(n-1)(k-1)+1}^{n(k-1)} X_i/\epsilon_{k-1}\}_{n\geq 1}$ which satisfies the same local and long-range dependence conditions as \z. For this stationary and max-stable sequence with unit Fréchet marginals, by applying (\ref{eiespk2}), we obtain
$$
\theta_{Z^*}=\frac{1}{1-E\left(e^{-(Z_1^*\vee Z_2^*)^{-1}}\right)}
-2.
$$
Then
$$
\theta_{Z}=\theta_{Z^*}=\frac{1}{1-E\left(e^{-(Z_1\vee Z_2)^{-1}\epsilon_{k-1}}\right)}
-2=\frac{1}{1-E\left(F_{Z}(Z_1)\vee F_{Z}(Z_2)\right)}-2.
$$
To obtain the relation (\ref{eimsx}) we apply Proposition \ref{peidep2} with
$$
\tau^*=\lim_{n\to\infty}nP(Z>n/\tau)=-\log F_Z(1)\tau.
$$
\end{proof}
This result suggests the estimation of $\theta_X$ via the estimation of $-\log F_Z(1)$ and $E\left(F_{Z}(Z_1)\vee F_{Z}(Z_2)\right)$.

\section{Estimation}\label{sestim}

Our new estimation proposal consists in first, to state the sequence of cycles, $Z_n=\bigvee_{s=(n-1)(k-1)+1}^{n(k-1)} X_s$, $n\geq 1$, and then  estimate $\theta$ based on  \z. Observe that, from Proposition \ref{peidep2}, we can define the estimator
\begin{eqnarray}\label{FDir}
\widehat{\theta}_X=\frac{U^Z_n(u_n)}{N^X_n(u_n)},
\end{eqnarray}
as well as, the estimator
\begin{eqnarray}\label{FInd}
\widehat{\theta}_X=\frac{\widehat{\theta}_Z N^Z_n(u_n)}{N^X_n(u_n)},
\end{eqnarray}
where $U^Z_n(u_n)$ and $N^Z_n(u_n)$ are, respectively, the number of upcrossings of $u_n$ and the number of exceedances of $u_n$  within $\{Z_1,\hdots,Z_{[n/(k-1)]}\}$  and $N^X_n(u_n)$ is the number of exceedances of $u_n$ within $\{X_1,\hdots,X_n\}$. Since D$^{(2)}$($u_n$) holds for \z, estimators under this condition can be used to calculate $\widehat{\theta}_Z$, e.g., the maximum likelihood in S\"{u}veges (\cite{suv07}, 2007) and the upcrossings estimator in Nandagopalan (\cite{nand90}, 1990).

Now observe that, based on (\ref{eitdcm}), we can write $\theta_Z$ as
\begin{eqnarray}\nn
\theta_Z=1-\lim_{n\to\infty}\frac{P(M_{I_{1},I_2}>u_n|M_{I_1}> u_n)}{P(M_{I_1}>u_n)}=1-\lambda_Z,
\end{eqnarray}
where $\lambda_Z$ is the so called ``tail dependence coefficient" (see Joe \cite{joe} 1997 p. 33, Coles \emph{et al.} \cite{coles+99} 1999, Schmidt and Stadm\"{u}ller \cite{schmi+stad06} 2006 and references therein; see also Ferreira and Ferreira \cite{fer+fer12a} 2012a Proposition 4). Hence, we can derive
\begin{eqnarray}\label{eitdc}
\theta_X=\frac{(1-\lambda_Z)\tau^*}{(k-1)\tau}.
\end{eqnarray}
and thus also state the estimator
\begin{eqnarray}\label{FIndtdc}
\widehat{\theta}_X=\frac{(1-\widehat{\lambda_Z})N^Z_n(u_n)}{N^X_n(u_n)},
\end{eqnarray}
We can estimate the tail dependence coefficient by applying a non-parametric procedure, e.g., the one in Schmidt and Stadm\"{u}ller (\cite{schmi+stad06}, 2006).
For the particular case of max-stable processes, by representation (\ref{eimsz}), we can also apply the estimators $\widehat{\theta}_Z$ in Ferreira and Ferreira (\cite{fer+fer12b}, 2012b) for $\theta_Z$ and again $\widehat{\theta}_X$ as in (\ref{FInd}). A similar procedure based on (\ref{eimsx}) leads to a second estimator for max-stable processes, namely
\begin{eqnarray}\nn
\widehat{\theta}_X=\widehat{\theta}_Z\frac{-\log \widehat{F}_Z(1)}{k-1},
\end{eqnarray}
where $\widehat{F}_Z(1)$ is the empirical distribution function. These estimators will be denoted, respectively, $\widehat{\theta}^{SS}$, $\widehat{\theta}^{FF}$ and $\widehat{\theta}^{FF^*}$.\\

In the next section we analyze our new proposal through simulation. For $\widehat{\theta}_Z$ in expression (\ref{FInd}), we consider the upcrossings estimator of Nandagopalan (\cite{nand90}, 1990), the estimator of Ferro and Segers (\cite{ferro+segers03}, 2003) also known as intervals estimator and the maximum likelihood estimator of S\"{u}veges (\cite{suv07}, 2007), and denote our extremal index estimators, respectively,  $\widehat{\theta}^{U}$,  $\widehat{\theta}^{I}$ and $\widehat{\theta}^{ML}$.
We also compare with the intervals and runs estimators applied directly on \x. For these estimators we use notation $\widetilde{\theta}^{I}$ and $\widetilde{\theta}^{R}$, respectively. \\

In order to analyze D$^{(k)}$($u_n$) and construct the cycles \z, we can extend the methodology in S\"{u}veges (\cite{suv07}, 2007) considered to check D$^{(2)}$($u_n$). More precisely, we compute the proportion of anti-D$^{(k)}$($u_n$) events by
$$
\dst p_k(u_n,r_n)=\frac{\sum_{j=1}^{n-r_n+1}\mathds{1}_{\{X_j>u_n,X_{j+1}\leq u_n,\hdots,X_{j+k-1}\leq u_n,M_{j+k-1,r_n+j-1}>u_n\}}}{\sum_{j=1}^{n}\mathds{1}_{\{X_j>u_n\}}},
$$
for normalized levels $u_n$ approximated by the empirical quantiles $1-\tau/n$, for some fixed positive $\tau$ and some sequence $\{r_n\}_{n\geq 1}$ satisfying the conditions of Proposition (\ref{peiindep}). We take the proportions $p_k(u_m,r_m)$ for sequences $\{X_1,\hdots,X_m\}$, with increasing length $m\leq n$.

On what concerns the choice of $\{k_n\}_{n\geq 1}$, we can choose, for instance, the family of sequences of integers $k_n^{(s)}=[(\log n)^{s}]$, $s>0$. Thus, for each $\tau$ and $s$, we can plot the points $(m,p_k(u_m,r_m^{(s)}))$, which must converge to zero, for some $s$, as $m$ increases if D$^{(k)}$($u_n$) holds with $k_n^{(s)}$. This is a slightly different approach of the one in S\"{u}veges (\cite{suv07}, 2007), but closer to the definition of D$^{(k)}$($u_n$), since this condition states a limiting behavior as $n\to\infty$, and $u_n\approx F^{-1}(1-\tau/n)$, $r_n=[n/k_n]$ and $p_k(u_n,r_n)$ are functions of $n$. To avoid three-dimensional plots that arise from the joint variation of $\tau$, $s$ and $m$, we can separately analyze the evolution of the proportions for different choices of $r_n$. For the particular case of D$^{(1)}$($u_n$) condition, we have the proportions
$$
p_1(u_n,r_n)=\frac{\sum_{j=1}^{n-r_n+1}\mathds{1}_{\{X_j>u_n,M_{j,r_n+j-1}>u_n\}}}
{\sum_{j=1}^{n}\mathds{1}_{\{X_j>u_n\}}}.
$$
Once accepted the condition D$^{(k)}$($u_n$) for some $k$, that means we consider that the process satisfies D$^{(s)}$($u_n$) for all $s\geq k$ and does not satisfy D$^{(s)}$($u_n$) for $s<k$. The decision to exclude values less than $k$ may be based on the analysis of $(m,p_{k-1}(u_m,r_m))$ or, from the remark after Proposition \ref{peidep2}, by comparing $d_{k-1}(u_n,r_n)$ with $d_{k}(u_n,r_n)$, where
$$
d_{k}(u_n,r_n)=\sum_{j=1}^{n-r_n+1}\mathds{1}_{\{X_j>u_n,M_{j,j+k-1}\leq u_n\}}.
$$
A good choice of $k$ is enhanced by away trajectories for $(m, d_{k-1}(u_m,r_m))$ and $(m, d_{k}(u_m,r_m))$ and close trajectories for $(m, d_{k}(u_m,r_m))$ and $(m, d_{k+1}(u_m,r_m))$.\\

An illustration is given in Figures \ref{figD3} and \ref{figD4-5}, where it was considered a simulated sample of size $10000$ for each of the models given below. In Figure \ref{figD3} it is plotted the proportions of anti-D$^{(3)}$($u_n$) (first five panels) and anti-D$^{(4)}$($u_n$) (last panel), with $k_n=[(\log n)^3]$ and values $\tau=50,100$, for some models. More precisely, the three first line panels correspond to the proportions of anti-D$^{(3)}$($u_n$) of the following models: a first order autoregressive process with Cauchy marginals and autoregressive parameter $\rho=-0.6$ of Chernick (\cite{chern78}, 1978), a negatively correlated uniform AR(1) process of Chernick \emph{et al.} (\cite{chern+91}, 1991) with $r=2$, respectively denoted ARCauchy and ARUnif, and an MM process with coefficients $\alpha_0=2/6,\,\alpha_1=1/6,\,\alpha_2=3/6$ as given in Example \ref{ex1}. The first second line panel correspond to the proportions of anti-D$^{(3)}$($u_n$)  of a first order MAR process with standard Fréchet marginals and autoregressive parameter $\phi=0.5$ of Davis and Resnick (\cite{dav+res89}, 1989). The last two panels in the second line correspond to,
respectively, the proportions of anti-D$^{(3)}$($u_n$) and anti-D$^{(4)}$($u_n$) of a Markov chain with standard Gumbel marginals and logistic joint distribution with dependence parameter $\alpha=0.5$
The MAR process satisfy D$^{(2)}$($u_n$) and so D$^{(3)}$($u_n$) holds, thus leading to small proportions of anti-D$^{(3)}$($u_n$). The same scenario is noticed in the first three models, all satisfying condition D$^{(3)}$($u_n$). There are slightly upper curves within the Markov chain  but still comprising small values.  A little decrease occurs in the proportions of anti-D$^{(4)}$($u_n$) for the Markov process.

In Figure \ref{figD4-5} we find the proportions of anti-D$^{(3)}$($u_n$) to anti-D$^{(5)}$($u_n$) of a GARCH(1,1) process with Gaussian innovations, autoregressive parameter $\lambda=0.25$ and variance parameter $\beta=0.7$ (Laurini and Tawn, \cite{lau+tawn12} 2012). More precisely, in the first two panels are plotted the proportions of anti-D$^{(3)}$($u_n$) by choosing $k_n=[(\log n)^3]$ and $k_n=[(\log n)^{3.3}]$, respectively, and the last two plots correspond to the proportions of anti-D$^{(4)}$($u_n$) and anti-D$^{(5)}$($u_n$), with  $k_n=[(\log n)^{3.3}]$. We can see that the choice $k_n=[(\log n)^{3.3}]$ may be better within this case. The plots also suggest that condition D$^{(3)}$($u_n$) is unlikely to hold for the considered GARCH(1,1) model. A more prominent decrease is observed within the proportions of anti-D$^{(4)}$($u_n$) and  anti-D$^{(5)}$($u_n$). It will be seen in the simulation study that these proportions lead to a quite acceptable choice of values of $k$ in the knew estimation procedure.

Observe that from Proposition \ref{peidep2.2} we can also analyze D$^{(k)}$($u_n$) by evaluating D$^{(2)}$($u_n$) within cycles \z. The respective plots are in Figures \ref{figD2Z} and \ref{figD2ZGarch} and seem to corroborate the analysis above.

The  plots observation can give us some clue about D$^{(k)}$($u_n$) but does not allow to make a definite decision. We can always opt for higher values of $k$ since, if D$^{(k)}$($u_n$) holds then D$^{(s)}$($u_n$) holds for all $s>k$. However, a too large $k$ for the cycles may diminish the precision of the new estimators, as will be pointed in the next section.

\begin{figure}
\begin{center}
\includegraphics[width=3.9cm,height=3.9cm]{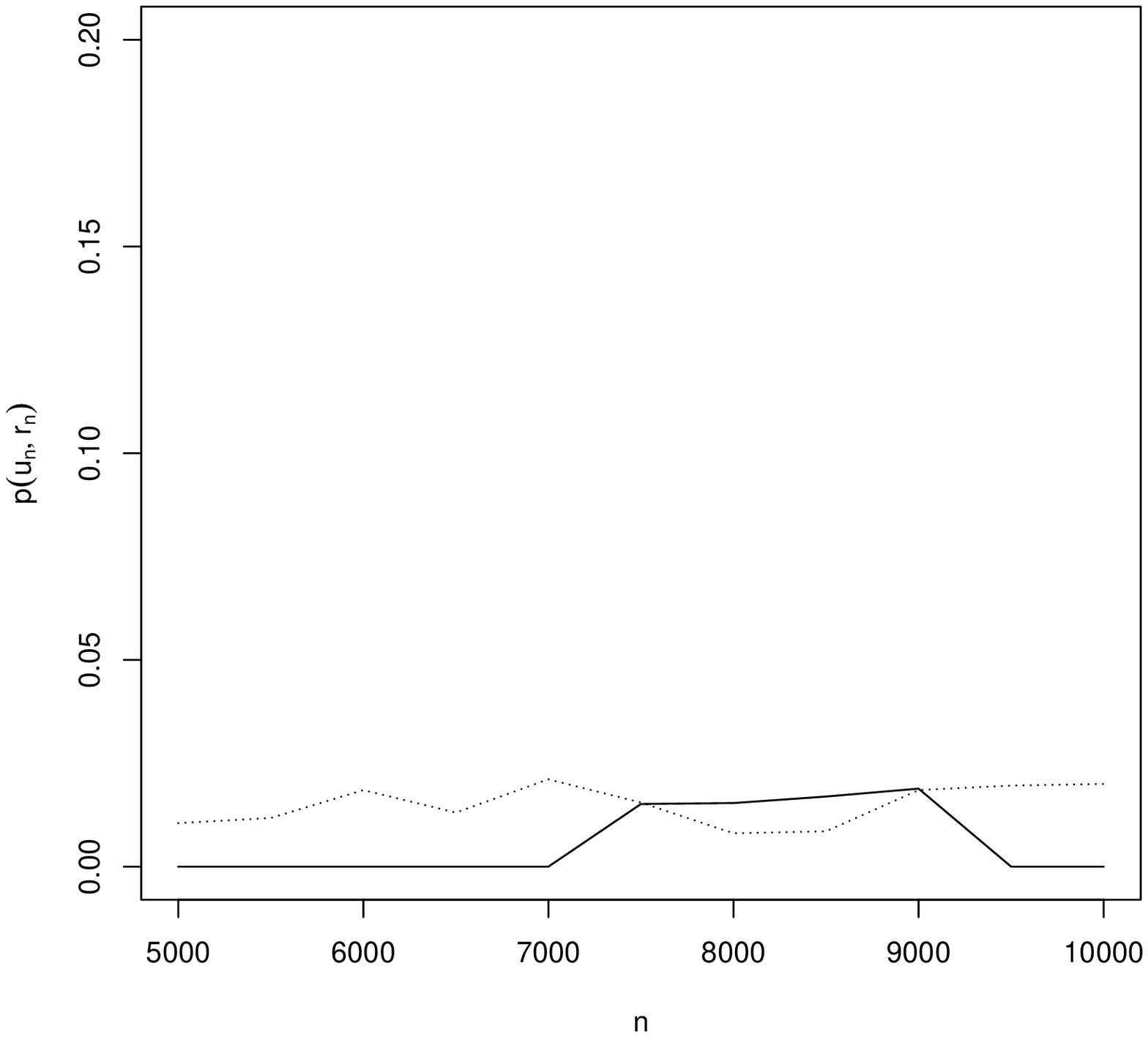}
\includegraphics[width=3.9cm,height=3.9cm]{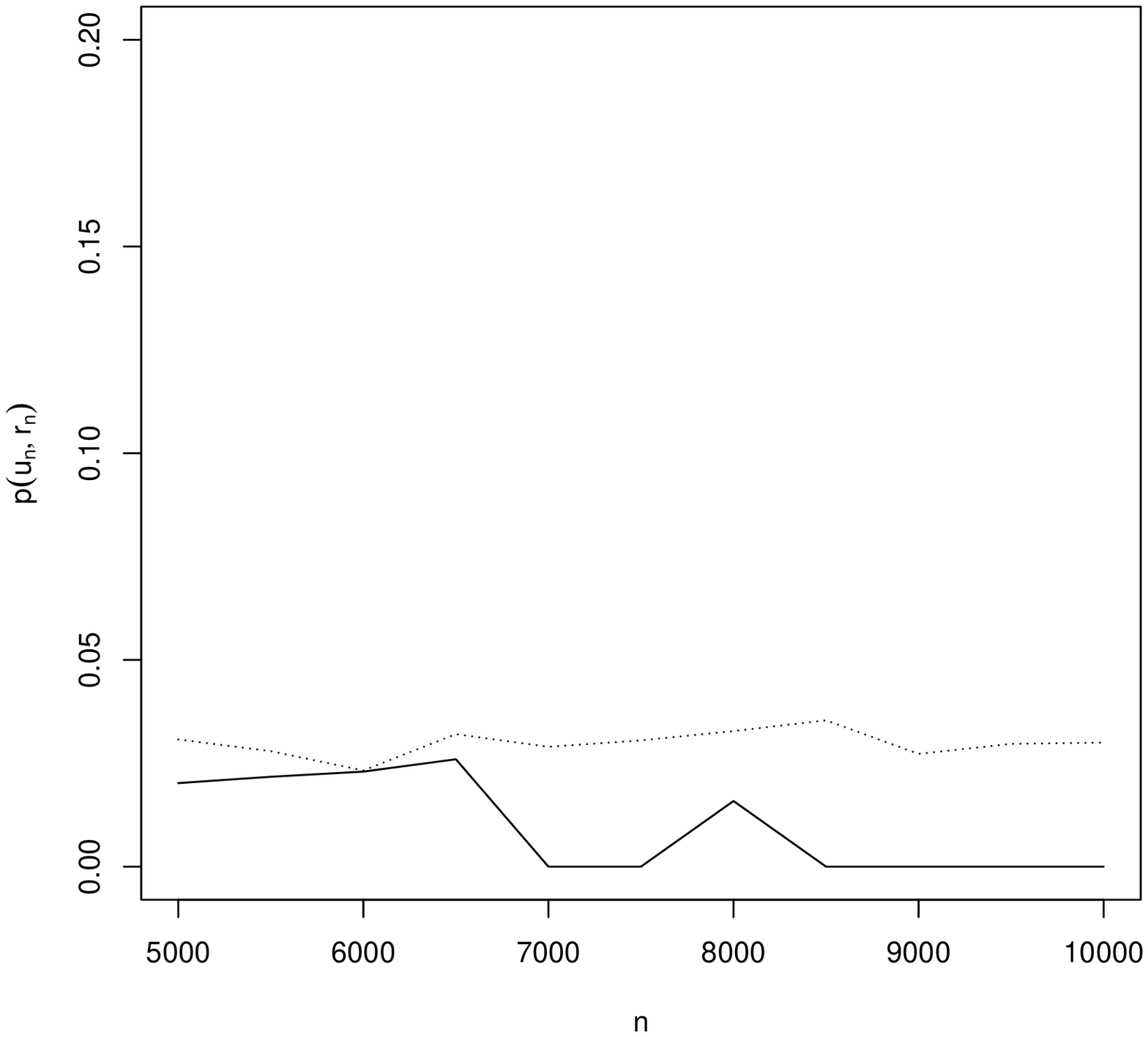}
\includegraphics[width=3.9cm,height=3.9cm]{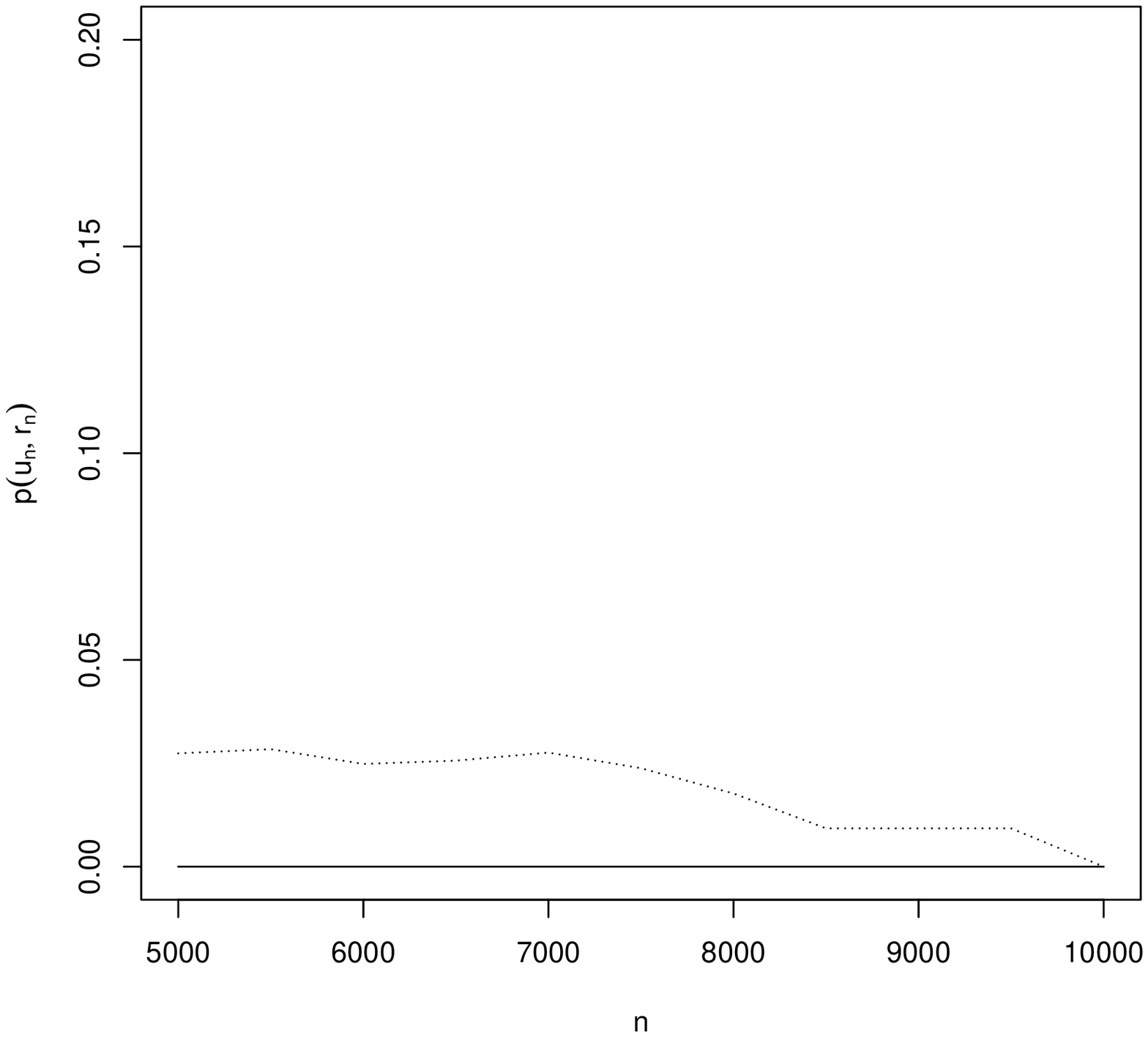}
\includegraphics[width=3.9cm,height=3.9cm]{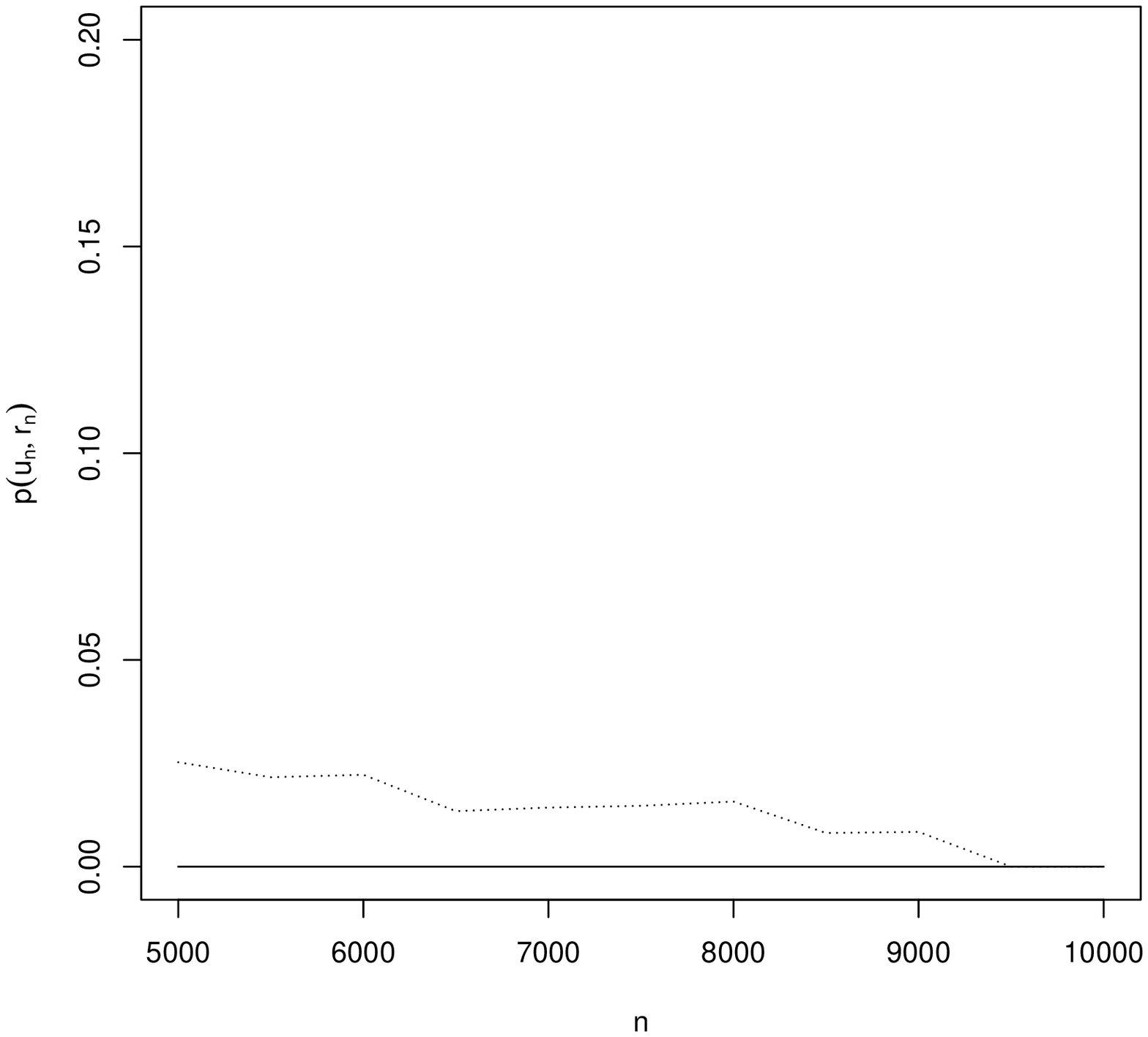}
\includegraphics[width=3.9cm,height=3.9cm]{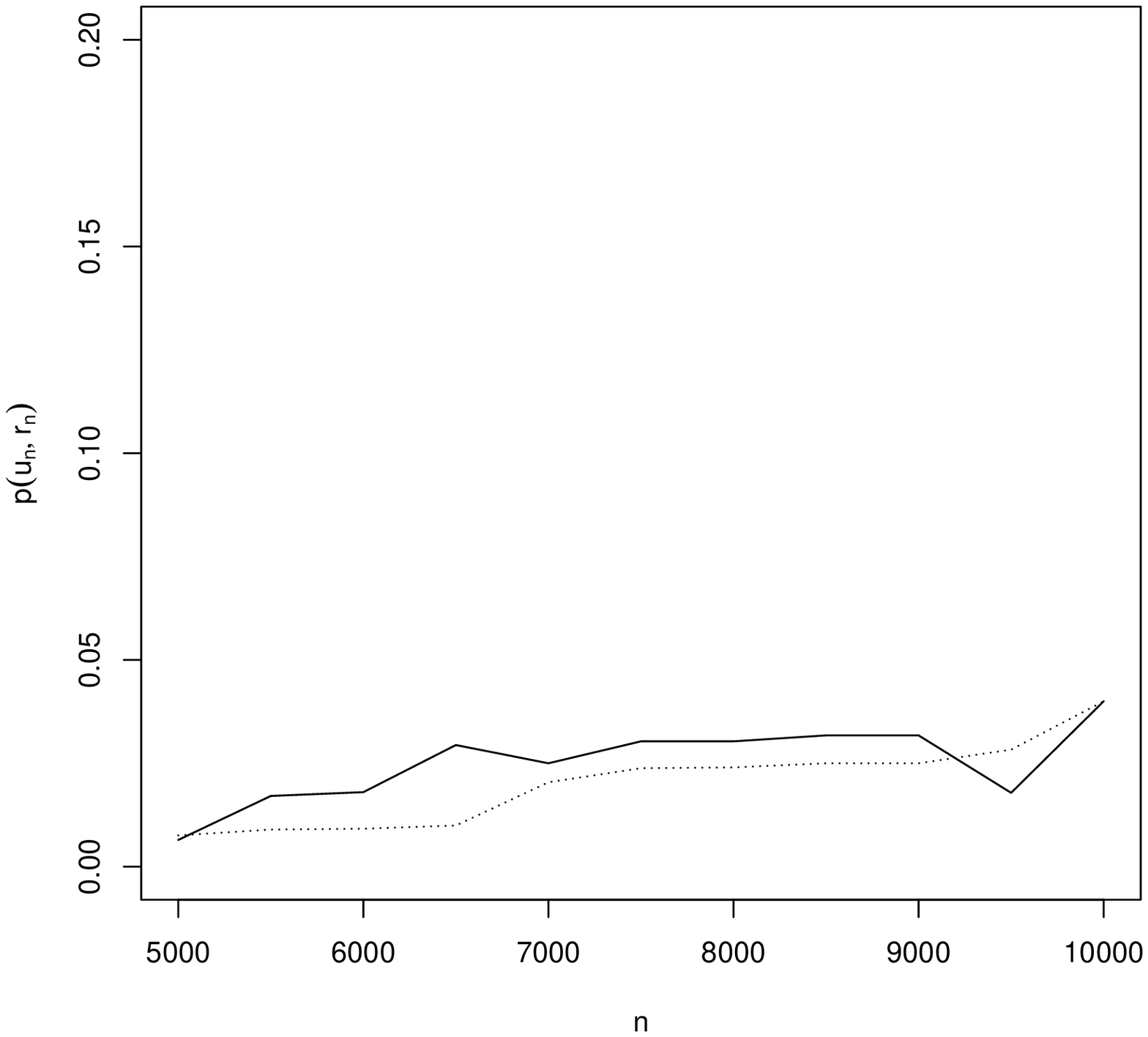}
\includegraphics[width=3.9cm,height=3.9cm]{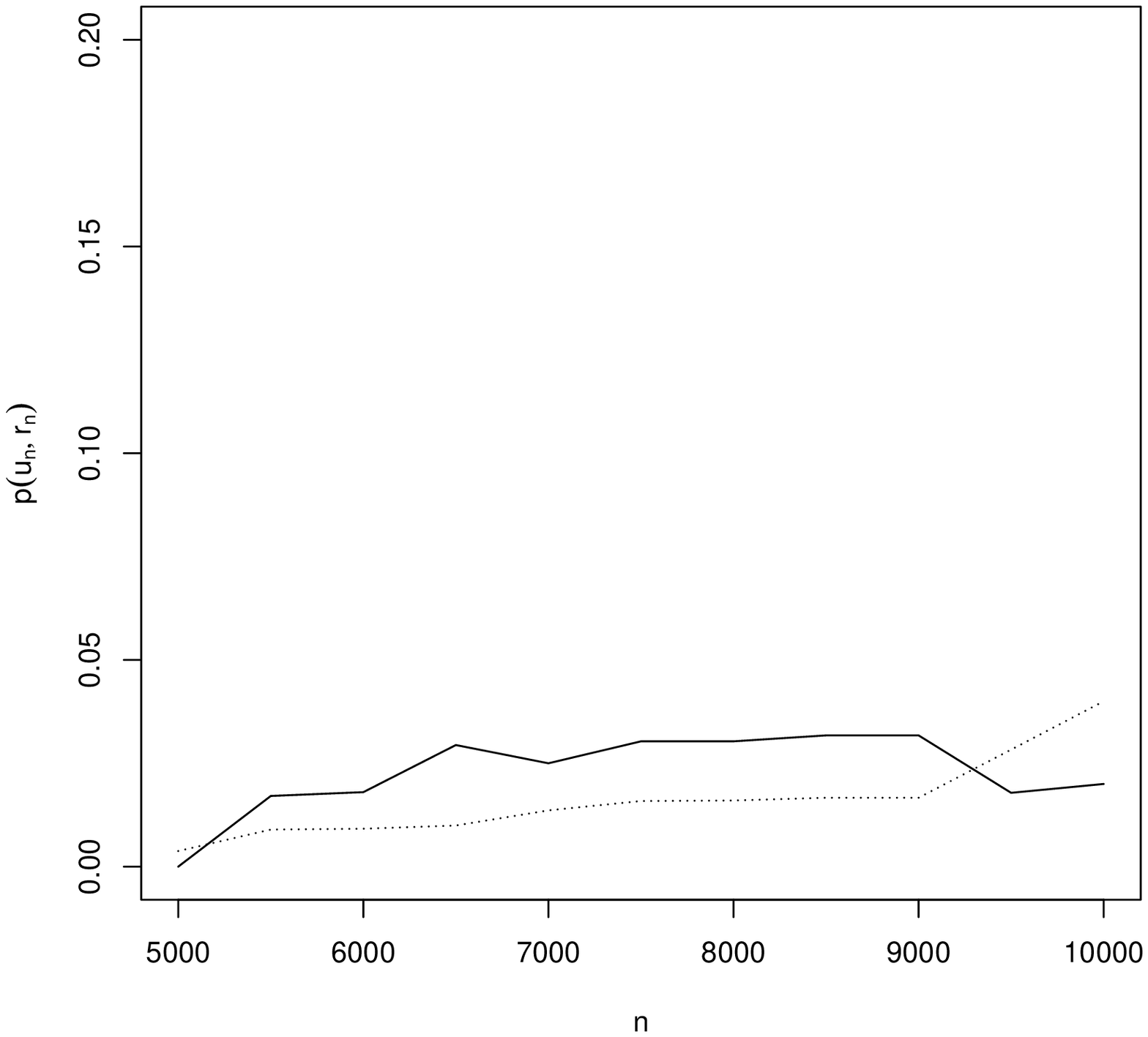}
\caption{From left to right and top to bottom, proportions of anti-D$^{(3)}$($u_n$) for ARCauchy, ARUnif, MM, MAR and Markov chain and anti-D$^{(4)}$($u_n$) of Markov chain, respectively, for $\tau=50$ (full line) and $\tau=100$ (dotted line), with $k_n=[(\log n)^3]$.\label{figD3}}
\end{center}
\end{figure}

\begin{figure}
\begin{center}
\includegraphics[width=3.9cm,height=3.9cm]{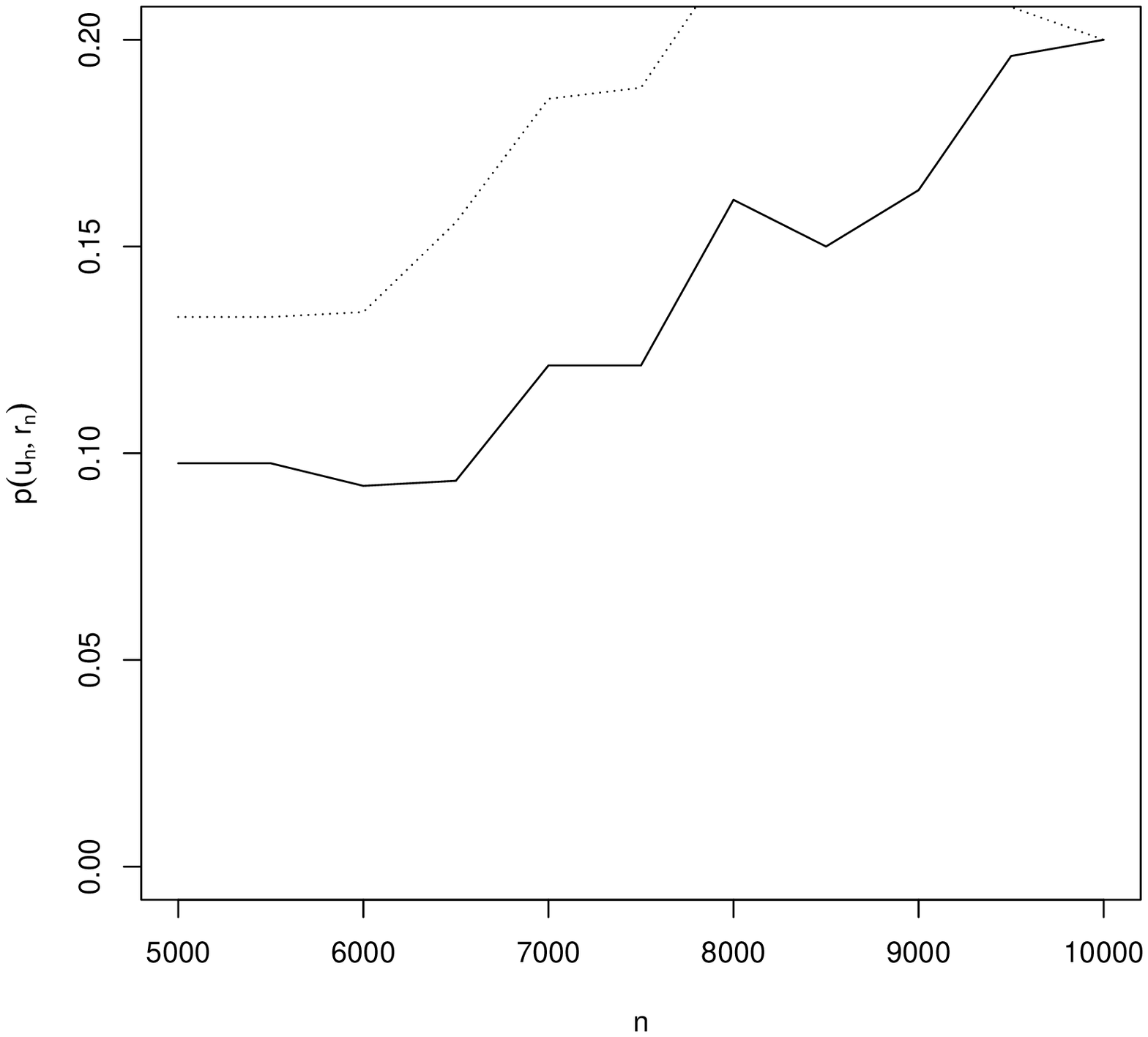}
\includegraphics[width=3.9cm,height=3.9cm]{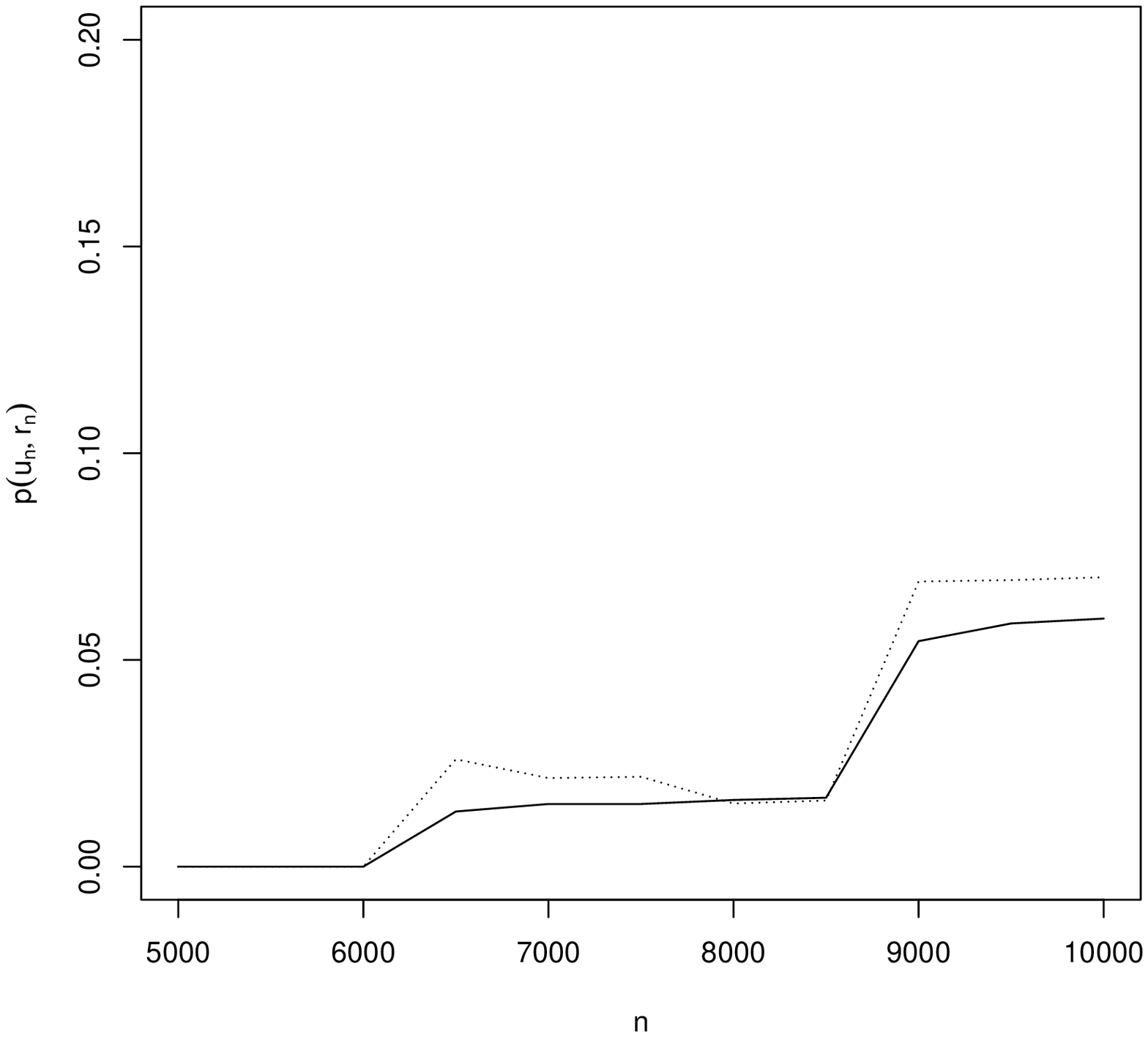}\\
\includegraphics[width=3.9cm,height=3.9cm]{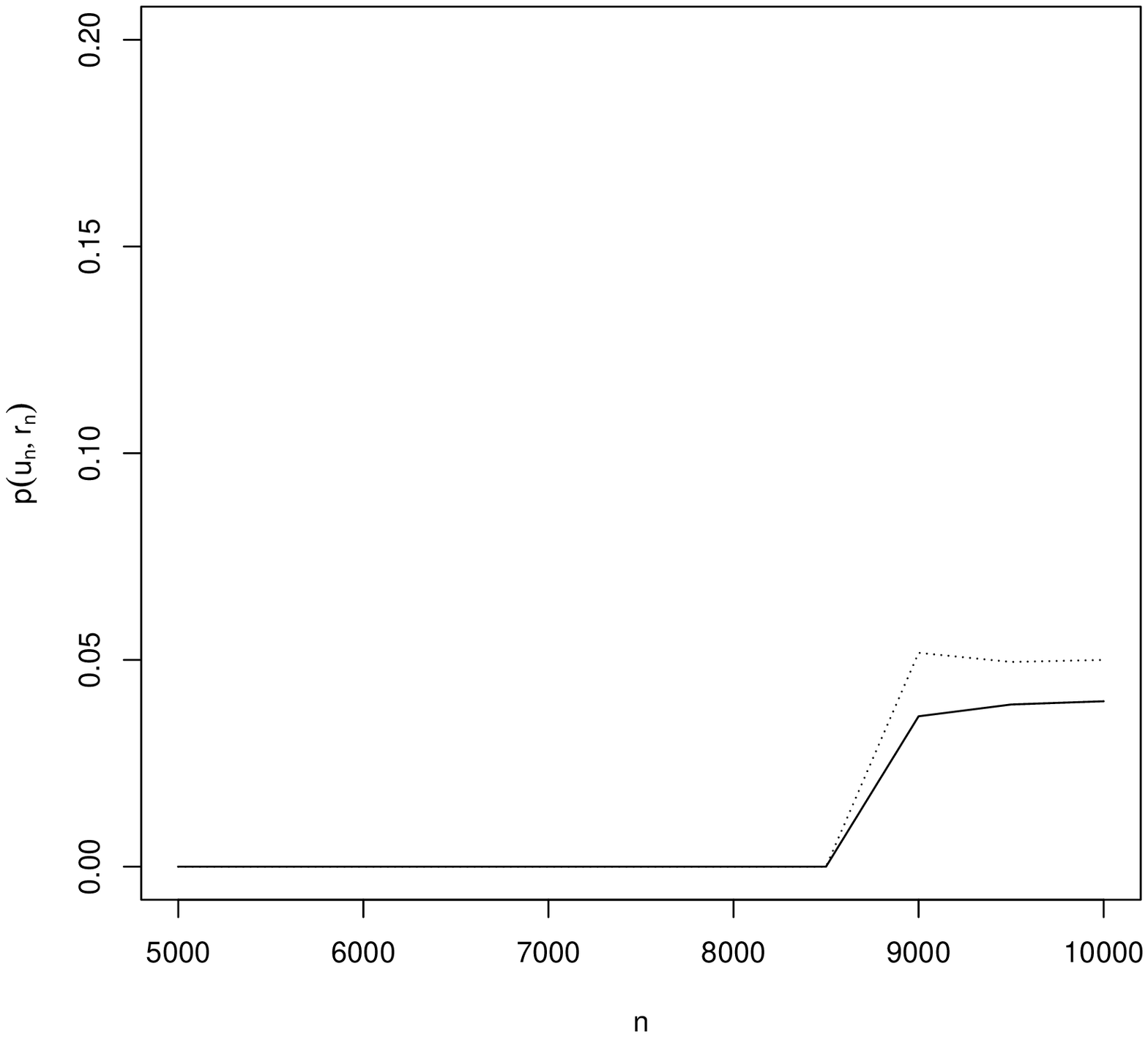}
\includegraphics[width=3.9cm,height=3.9cm]{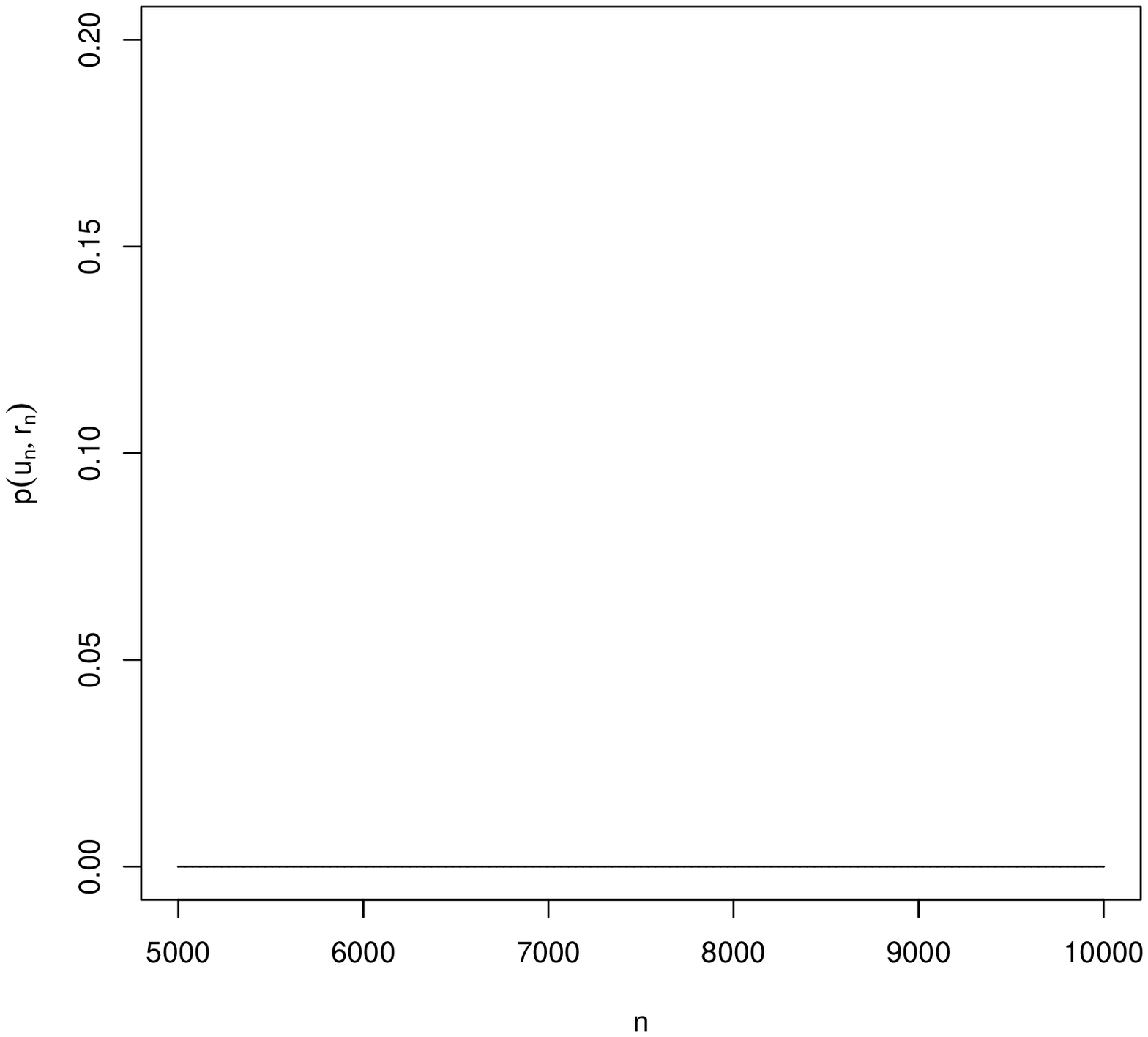}
\caption{From left to right and top to bottom, proportions of anti-D$^{(3)}$($u_n$) with $k_n=[(\log n)^3]$ and anti-D$^{(3)}$($u_n$) to anti-D$^{(5)}$($u_n$) with $k_n=[(\log n)^{3.3}]$ for   GARCH(1,1), for $\tau=50$ (full line) and $\tau=100$ (dotted line).\label{figD4-5}}
\end{center}
\end{figure}

\begin{figure}
\begin{center}
\includegraphics[width=3.9cm,height=3.9cm]{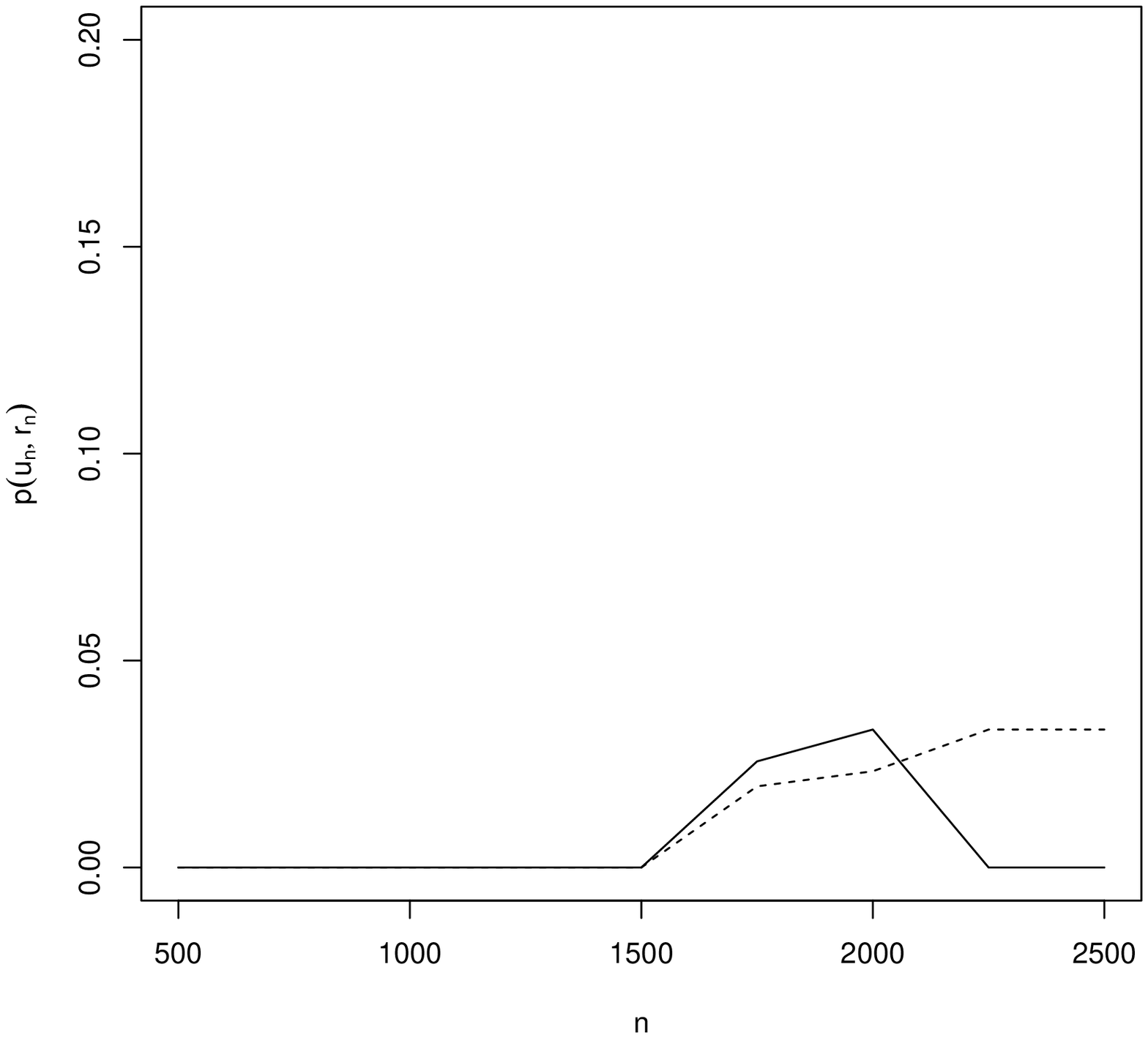}
\includegraphics[width=3.9cm,height=3.9cm]{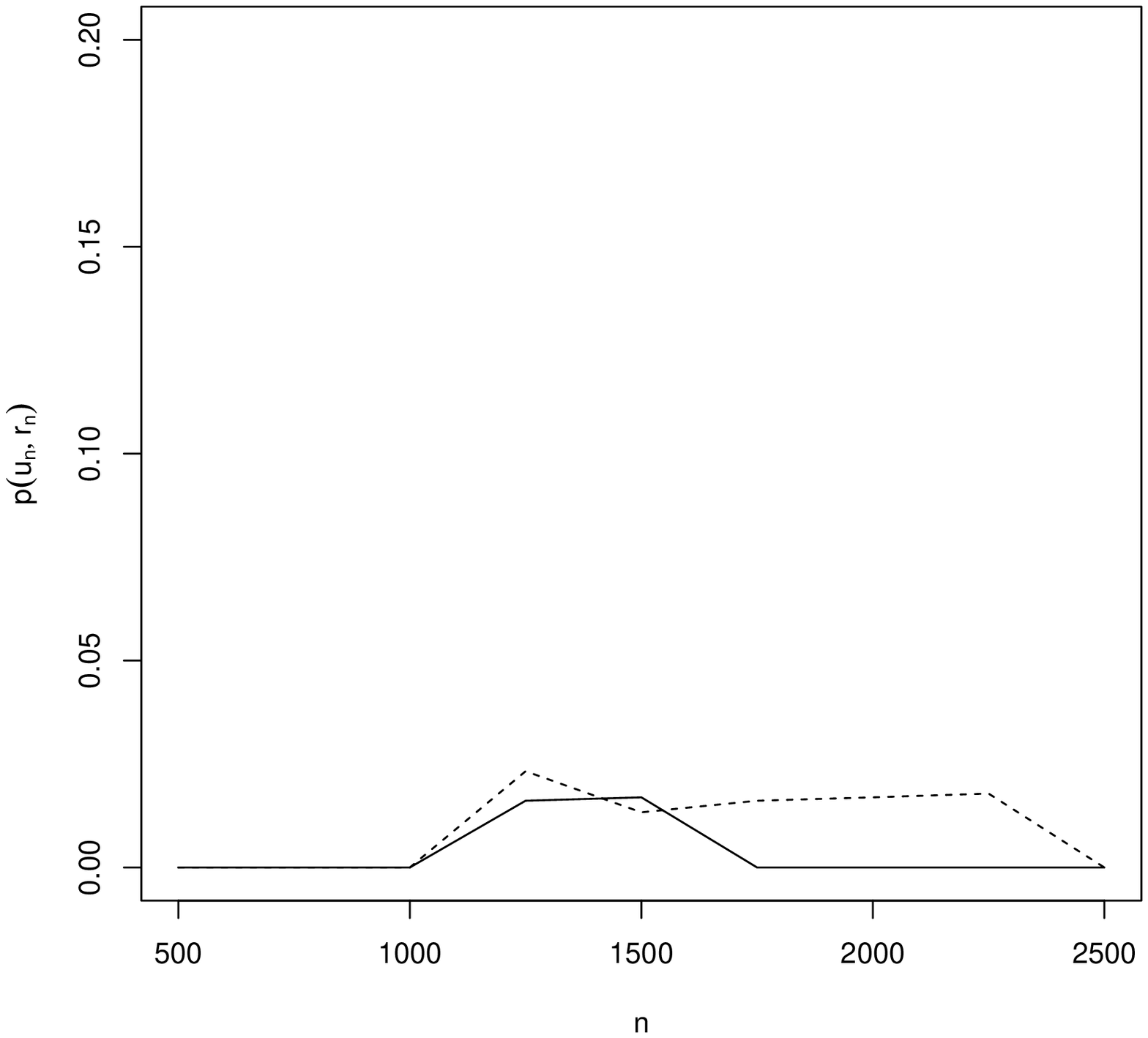}
\includegraphics[width=3.9cm,height=3.9cm]{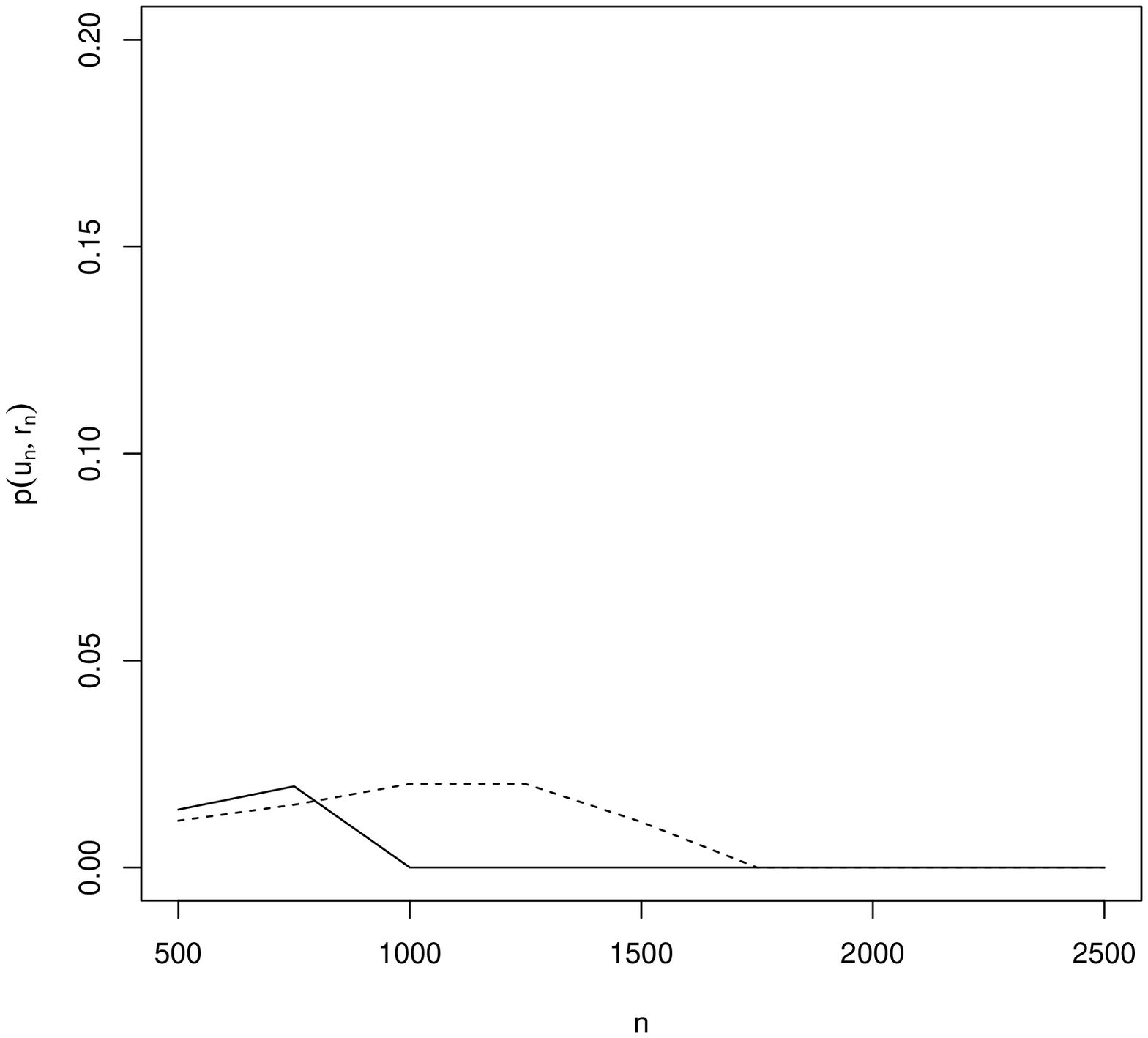}
\includegraphics[width=3.9cm,height=3.9cm]{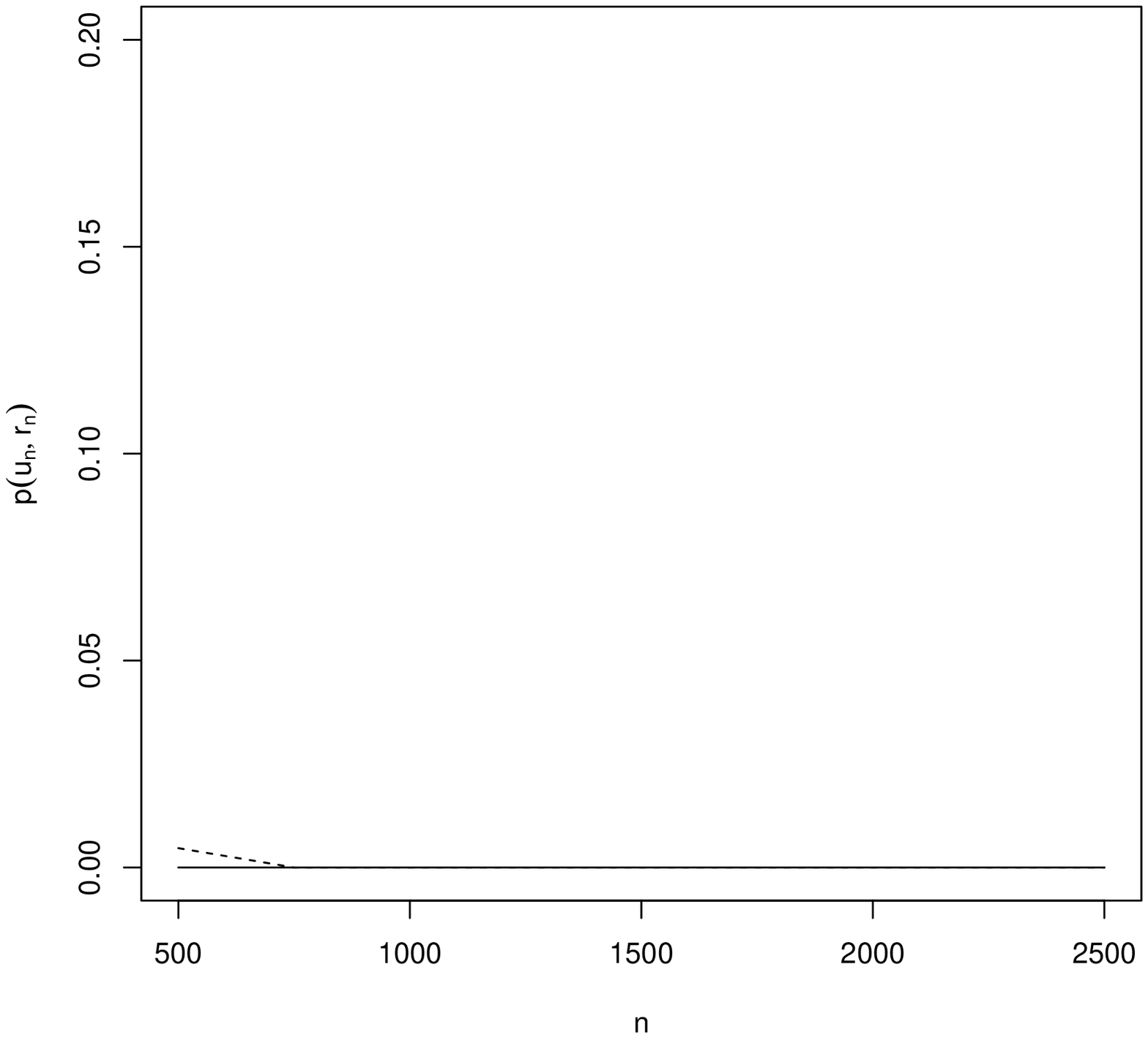}
\includegraphics[width=3.9cm,height=3.9cm]{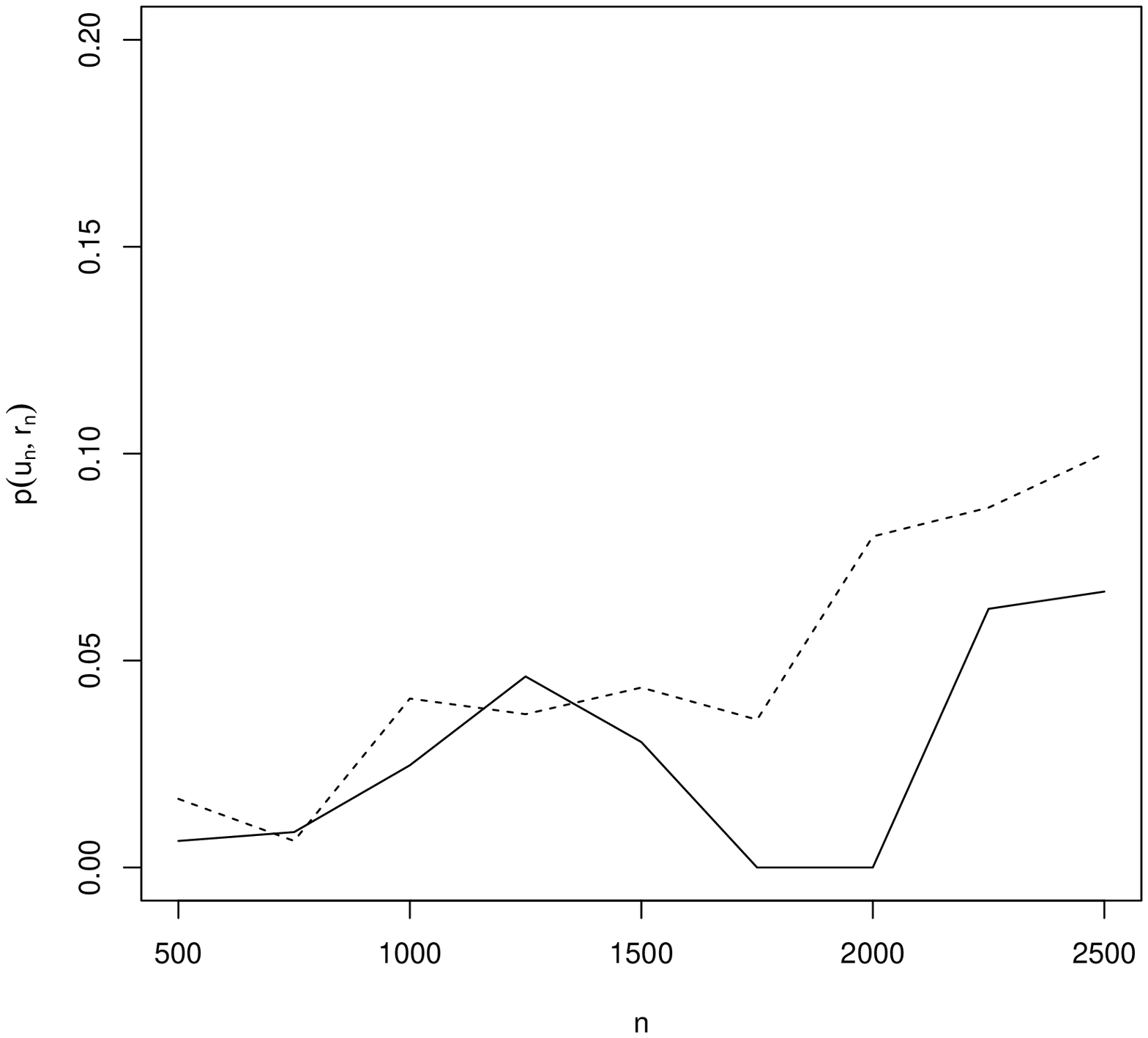}
\includegraphics[width=3.9cm,height=3.9cm]{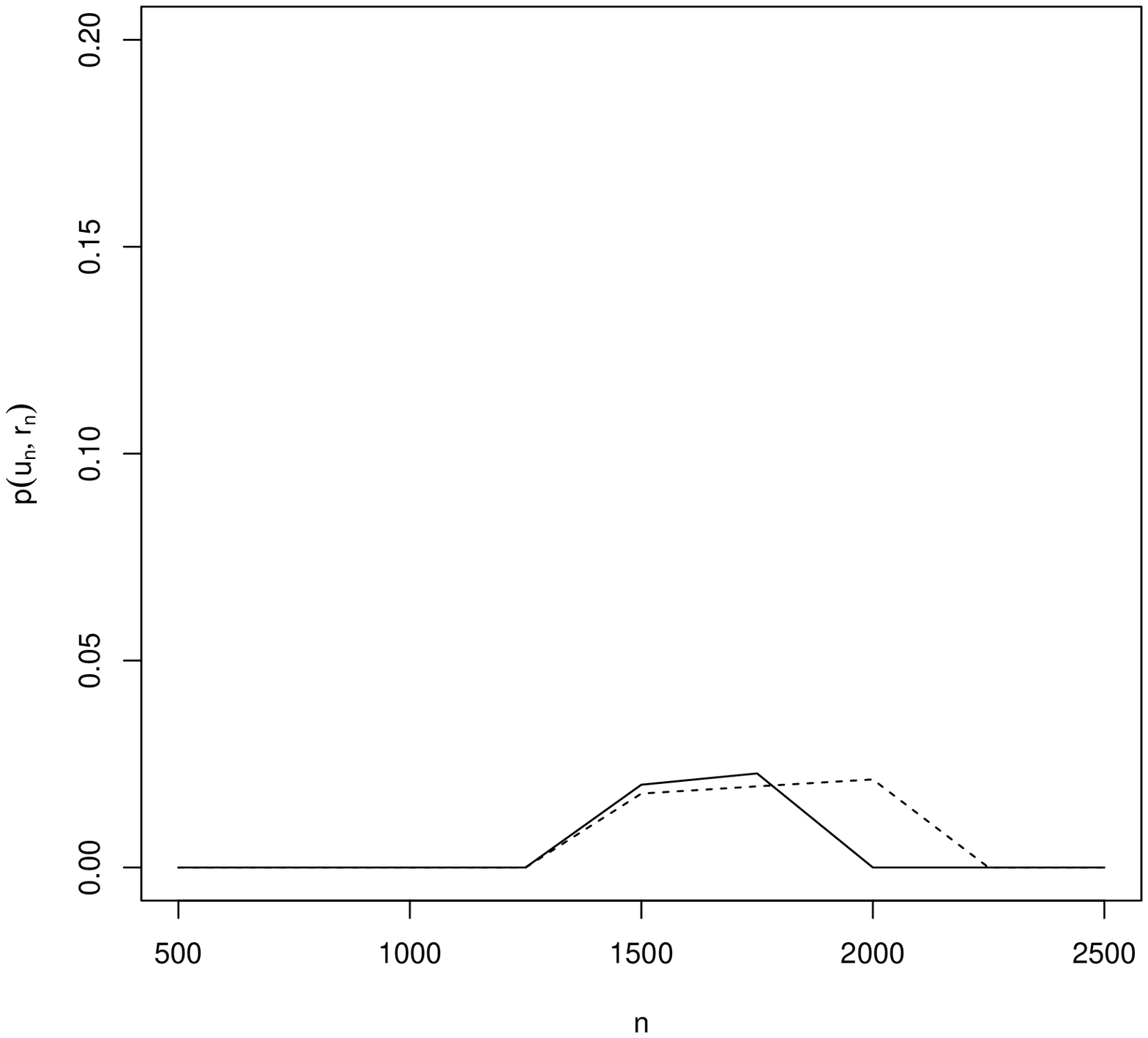}
\caption{From left to right and top to bottom, proportions of anti-D$^{(2)}$($u_n$) of cycles \z\esp for ARCauchy, ARUnif, MM, MAR and Markov chain with $k=3$ and  Markov chain with $k=4$, for $\tau=15$ (full line) and $\tau=20$ (dotted line), with $k_n=[(\log n)^3]$.\label{figD2Z}}
\end{center}
\end{figure}

\begin{figure}
\begin{center}
\includegraphics[width=3.9cm,height=3.9cm]{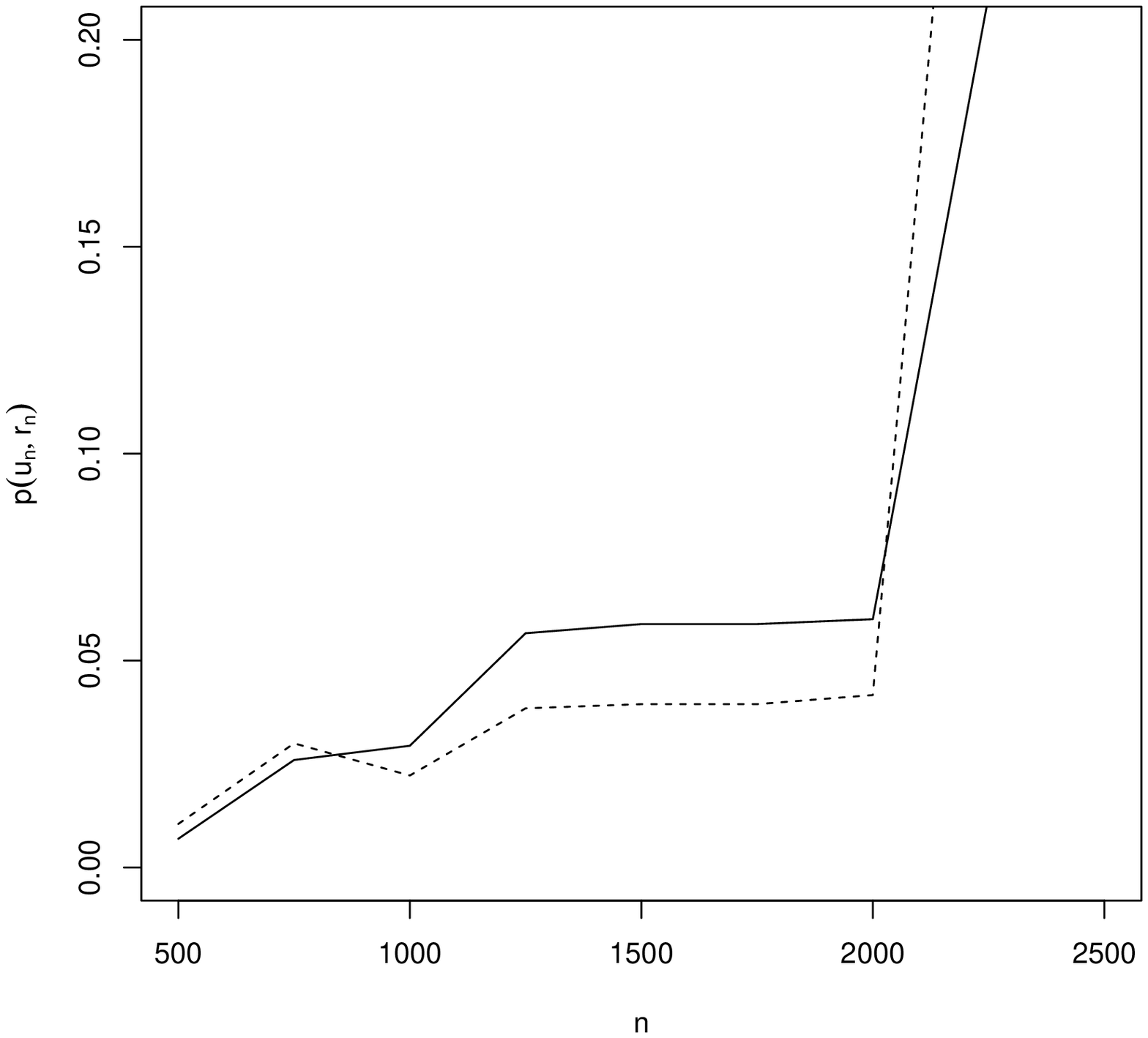}
\includegraphics[width=3.9cm,height=3.9cm]{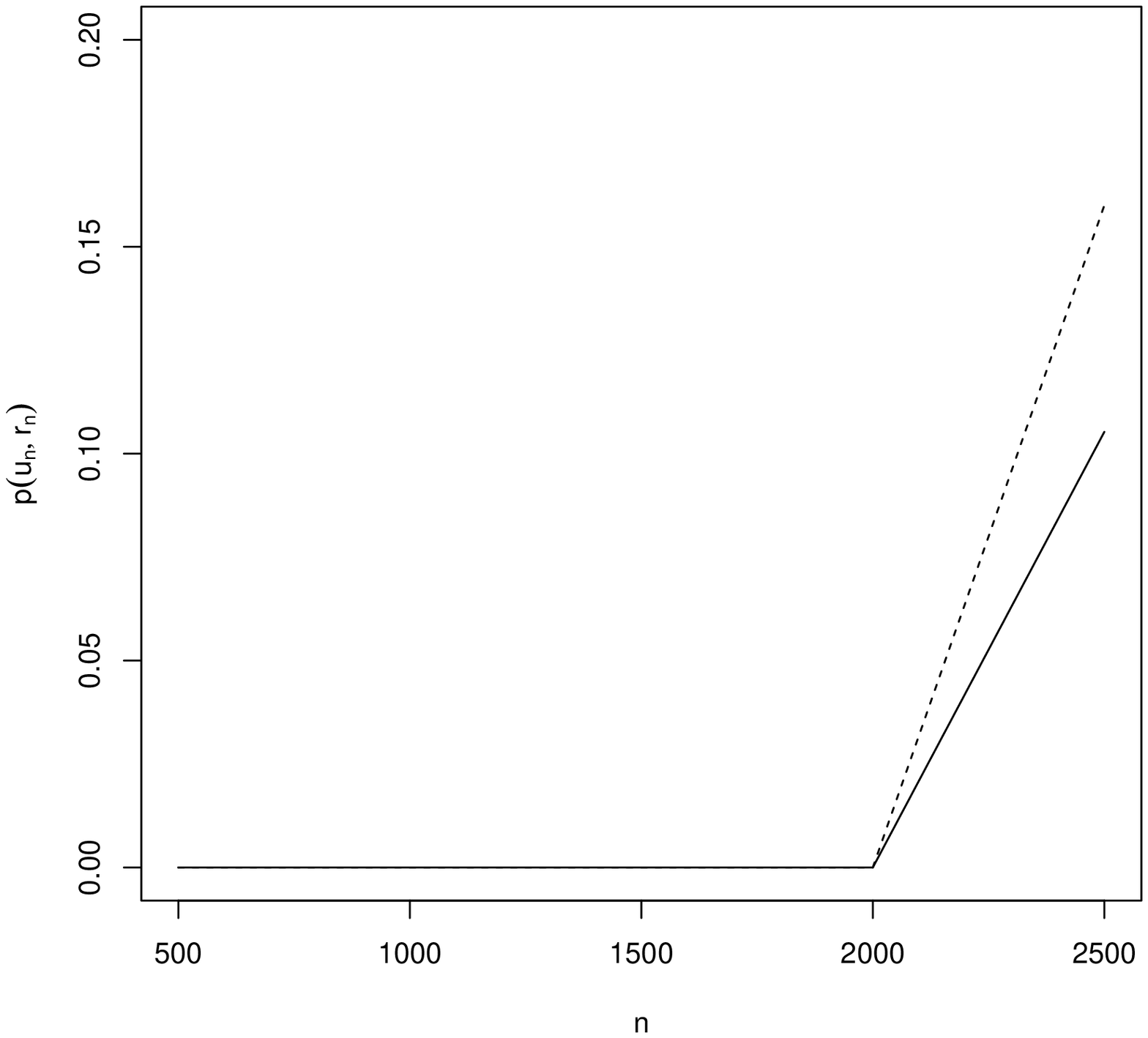}\\
\includegraphics[width=3.9cm,height=3.9cm]{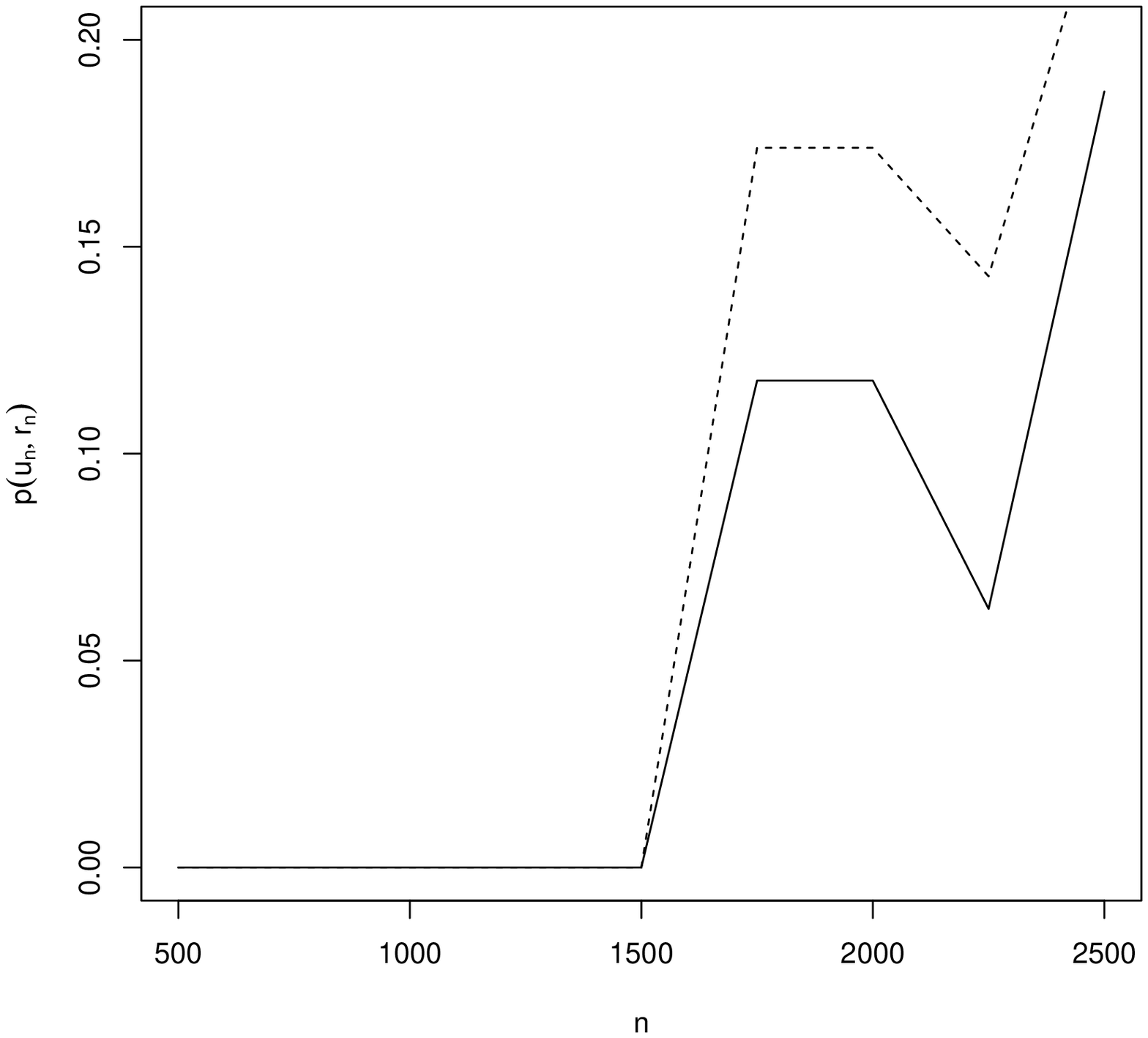}
\includegraphics[width=3.9cm,height=3.9cm]{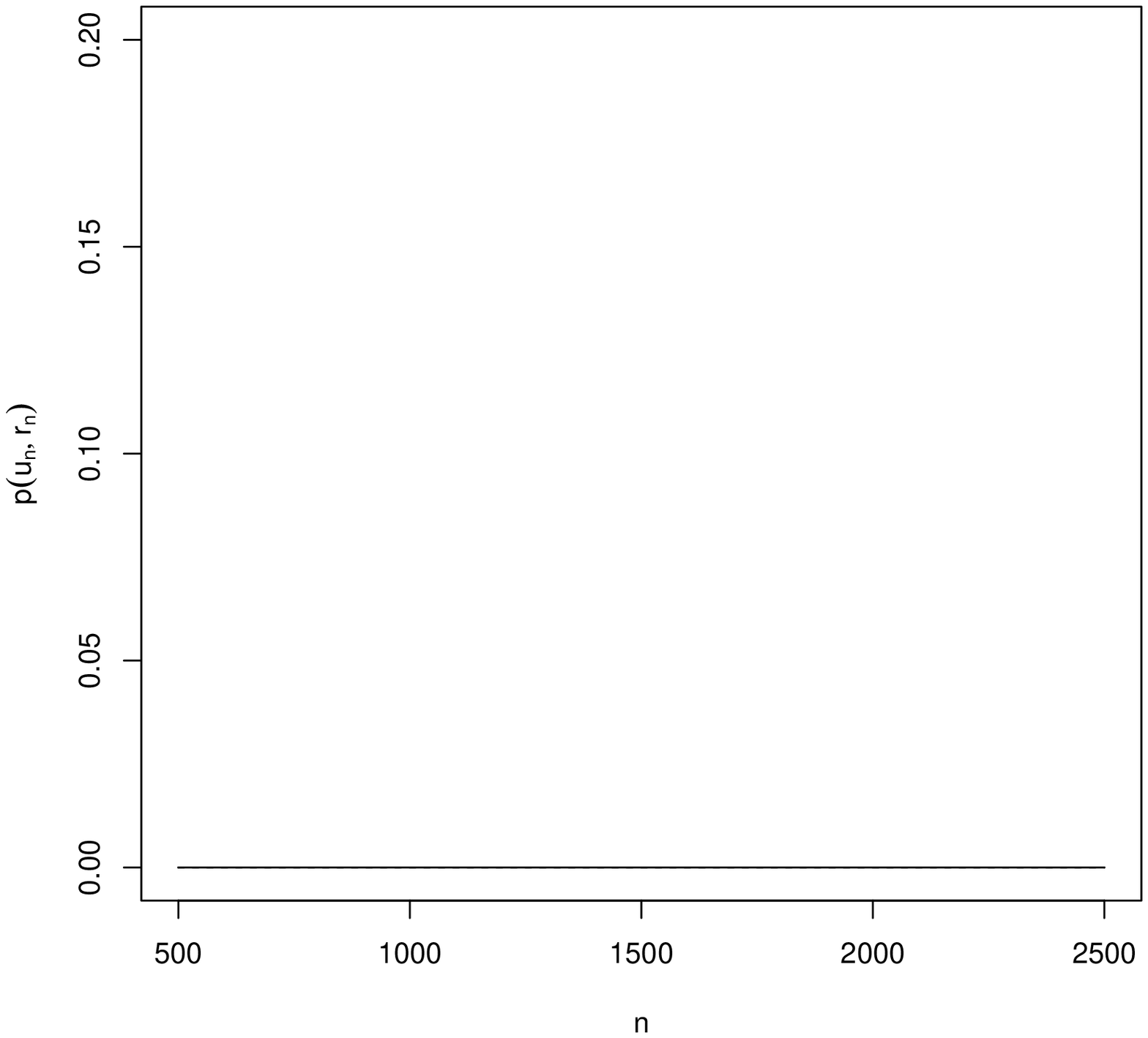}
\caption{From left to right and top to bottom, proportions of anti-D$^{(2)}$($u_n$) of cycles \z\esp for GARCH(1,1) with $k=3$ and $k_n=[(\log n)^{2.5}]$ and  GARCH(1,1) with $k=3,4,5$ and $k_n=[(\log n)^3]$, for $\tau=15$ (full line) and $\tau=20$ (dotted line).\label{figD2ZGarch}}
\end{center}
\end{figure}

\subsection{Simulations}\label{ssim}
In our study we consider $1000$ replicates of simulated samples of size $1000$ of each of the models referred previously: ARCauchy ($\rho=-0.6$), ARUnif ($r=2$), MM ($\alpha_0=2/6,\alpha_1=1/6,\alpha_2=3/6$), MAR ($\phi=0.5$), Markov chain ($\alpha=0.5$) and GARCH(1,1) ($\lambda=0.25$, $\beta=0.7$). We have calculated the values of the new estimator $\widehat{\theta}$ given in (\ref{FDir}), as well as, the values of estimators  $\widehat{\theta}^{I}$, $\widehat{\theta}^{ML}$ and $\widehat{\theta}^{U}$ based on the new indirect approach in (\ref{FInd}), and estimator $\widehat{\theta}^{SS}$ based on (\ref{FIndtdc}). Although   $\widehat{\theta}^{FF}$ and $\widehat{\theta}^{FF^*}$  are derived under a max-stable premise, we still apply them since, in practice,  we are taking cycles \z\esp of maximums which, albeit crudely, can approach a max-stable behavior. We denote all these estimators as indirect. For comparison, we also consider the runs estimator ($\widetilde{\theta}^{R}$)  and the intervals estimator ($\widetilde{\theta}^{I}$) directly for $\theta_X$. In opposition to indirect estimators, we denote $\widetilde{\theta}^{R}$ and  $\widetilde{\theta}^{I}$ as direct.
The root mean squared errors (rmse) and the absolute mean biases (abias) are given in Tables \ref{tab1} and \ref{tab2}, for levels $u_n$ corresponding to the empirical quantiles $0.95$, $0.975$ and $0.99$, respectively denoted, $q_{0.95}$, $q_{0.975}$ and $q_{0.99}$.
For models MM, ARUnif, ARCauchy and MAR, which satisfy condition D$^{(3)}$($u_n$), all the new estimators were based on the construction of cycles \z\esp by taking $k=3$. The direct runs estimator $\widetilde{\theta}^{R}$ was also computed for run $r=k$ (see Section \ref{sint}). These are reported in Table \ref{tab1}.
The results for the Markov chain and GARCH(1,1) are given in Table \ref{tab2}, by considering that D$^{(k)}$($u_n$) holds with $k=4$ in the first case and $k=5$ in the second model (see Figures \ref{figD3} to \ref{figD2ZGarch}). In what concerns the direct runs estimator $\widetilde{\theta}^{R}$, we choose a run $r$ equal to $k=4$   in the Markov chain model and run $r$ equal to $k=5$ in the GARCH(1,1) model. Ancona-Navarrete and Tawn (\cite{anc+tawn00}, 2000) considered $r=10$ for the runs estimator $\widetilde{\theta}^{R}$ in the Markov chain model. Indeed, if we take $r=10$ in our simulations for this model, we obtain slightly lower rmse's for this estimator. We have also considered $r=10$ in the GARCH(1,1) model which led to an overall decreasing of $0.1$ in the rmse's for estimator $\widetilde{\theta}^{R}$.
The presented choice of the values $k$ for the indirect estimators leads to the best results among other values of $k$ also tried in simulations but not reported. Indeed, if the models satisfy condition D$^{(k)}$($u_n$), the results by taking  $k+1$ are quite close but if we continue to increase $k$, they get worst. Observe that a too much large $k$ means larger cycles \z\esp and thus some loss of information.

The new approach presents good results, particularly with  estimators $\widehat{\theta}$ and $\widehat{\theta}^{U}$. As expected, the upcrossings estimator is a competitor within our framework.  The estimator $\widehat{\theta}^{FF}$ has also a good performance, except for the ARUnif model. In this case the results are better if we take $k=4$, leading to a rmse ranging from $0.084$ to $0.158$. One reason is that the cycles \z\esp with $k=4$ for this model may be more close of a max-stable behavior. We observe a similar situation with estimator $\widehat{\theta}^{FF^*}$. It performs well except in model ARUnif  where, for $k=4$, we obtain a rmse of 0.077, as well as in model ARCauchy where $k=4$ leads to a rmse of 0.063.
The intervals estimator yields the largest errors and behaves better if applied indirectly in the case of the Markov chain and the GARCH(1,1).
The indirect estimators $\widehat{\theta}^{ML}$ and $\widehat{\theta}^{SS}$ have a similar performance. The results tend to be better at lower thresholds.

\begin{table}
\caption{The root mean squared error (rmse) and the absolute mean bias (abias) obtained for models MM, ARUnif, ARCauchy and MAR, by considering the empirical quantiles $0.95$, $0.975$ and $0.99$, respectively, $q_{0.95}$, $q_{0.975}$ and $q_{0.99}$. The direct runs estimator $\widetilde{\theta}^{R}$ is based on run $r=3$.
\label{tab1}}
\begin{center}
\begin{tabular}{|l|ccc|ccc|}
\hline
&&rmse&&&abias&\\
\hline
%
%
MM&  $q_{0.95}$ &   $q_{0.975}$    &   $q_{0.99}$ &   $q_{0.95}$ &   $q_{0.975}$    &   $q_{0.99}$ \\
\hline

$\widetilde{\theta}^{R}$&   0.055   &   0.063   &   0.095   &   0.028   &   0.002   &   0.024   \\
$\widetilde{\theta}^{I}$&   0.114   &   0.152   &   0.221   &   0.061   &   0.075   &   0.137   \\
$\widehat{\theta}$ &   0.057   &   0.062   &   0.095   &   0.036   &   0.007   &   0.020    \\
$\widehat{\theta}^{U}$&   0.055   &   0.077   &   0.138   &   0.013   &   0.011   &   0.058   \\
$\widehat{\theta}^{I}$&   0.141   &   0.184   &   0.268   &   0.071   &   0.105   &   0.211   \\
$\widehat{\theta}^{ML}$&   0.063   &   0.077   &   0.138   &   0.009   &   0.008   &   0.031   \\
$\widehat{\theta}^{SS}$&   0.055   &   0.084   &   0.176   &   0.000   &   0.023   &   0.112   \\
$\widehat{\theta}^{FF}$&   0.032   &   0.032   &   0.055   &   0.003   &   0.004   &   0.014   \\

\hline
$\widehat{\theta}^{FF^*}$&     &   0.032   &      &     &   0.003   &     \\
\hline
\hline
ARUnif  &   $q_{0.95}$ &   $q_{0.975}$     &   $q_{0.99}$ &   $q_{0.95}$ &   $q_{0.975}$     &   $q_{0.99}$ \\
\hline

$\widetilde{\theta}^{R}$&   0.063   &   0.089   &   0.138   &   0.005   &   0.011   &   0.021   \\
$\widetilde{\theta}^{I}$&   0.179   &   0.130    &   0.118   &   0.200 &   0.182   &   0.202   \\
$\widehat{\theta}$   &   0.003   &   0.009   &   0.019   &   0.063   &   0.089   &   0.138   \\
$\widehat{\theta}^{U}$&   0.089   &   0.118   &   0.182   &   0.011   &   0.018   &   0.039   \\
$\widehat{\theta}^{I}$ &   0.130    &   0.120    &   0.145   &   0.182   &   0.195   &   0.219   \\
$\widehat{\theta}^{ML}$&   0.015   &   0.017   &   0.022   &   0.089   &   0.122   &   0.197   \\
$\widehat{\theta}^{SS}$&   0.020    &   0.025   &   0.088   &   0.089   &   0.118   &   0.179   \\
$\widehat{\theta}^{FF}$&   0.335   &   0.335   &   0.335   &   0.331   &   0.331   &   0.331   \\
\hline
$\widehat{\theta}^{FF^*}$&     &   0.875   &      &      &   0.861   &      \\
\hline
\hline
ARCauchy
 &   $q_{0.95}$ &   $q_{0.975}$     &   $q_{0.99}$ &   $q_{0.95}$ &   $q_{0.975}$     &   $q_{0.99}$ \\
\hline

$\widetilde{\theta}^{R}$&   0.077   &   0.095   &   0.152   &   0.041   &   0.013   &   0.026   \\
$\widetilde{\theta}^{I}$&   0.158   &   0.182   &   0.237   &   0.095   &   0.089   &   0.132   \\
$\widehat{\theta}$&   0.084   &   0.095   &   0.152   &   0.051   &   0.019   &   0.022   \\
$\widehat{\theta}^{U}$&   0.095   &   0.130    &   0.210    &   0.018   &   0.006   &   0.068   \\
$\widehat{\theta}^{I}$&   0.179   &   0.219   &   0.286   &   0.088   &   0.112   &   0.194   \\
$\widehat{\theta}^{ML}$&   0.095   &   0.134   &   0.219   &   0.014   &   0.006   &   0.05    \\
$\widehat{\theta}^{SS}$&   0.089   &   0.134   &   0.219   &   0.003   &   0.026   &   0.141   \\
$\widehat{\theta}^{FF}$&   0.084   &   0.084   &   0.084   &   0.072   &   0.074   &   0.075   \\
\hline
$\widehat{\theta}^{FF^*}$&      &   0.602   &      &    &   0.595   &      \\
\hline
\hline
MAR
&   $q_{0.95}$ &   $q_{0.975}$     &   $q_{0.99}$ &   $q_{0.95}$ &   $q_{0.975}$     &   $q_{0.99}$ \\
\hline
$\widetilde{\theta}^{R}$&   0.071   &   0.095   &   0.158   &   0.005   &   0.017   &   0.058   \\
$\widetilde{\theta}^{I}$&   0.134   &   0.176   &   0.261   &   0.067   &   0.087   &   0.157   \\
$\widehat{\theta}$&   0.071   &   0.094   &   0.154   &   0.026   &   0.007   &   0.051   \\
$\widehat{\theta}^{U}$&   0.077   &   0.114   &   0.187   &   0.009   &   0.022   &   0.072   \\
$\widehat{\theta}^{I}$&   0.145   &   0.184   &   0.251   &   0.075   &   0.105   &   0.178   \\
$\widehat{\theta}^{ML}$&   0.077   &   0.114   &   0.192   &   0.005   &   0.021   &   0.055   \\
$\widehat{\theta}^{SS}$&   0.077   &   0.114   &   0.210    &   0.002   &   0.032   &   0.126  \\
$\widehat{\theta}^{FF}$&   0.045   &   0.055   &   0.077    &   0.003   &   0.009   &   0.023  \\
\hline
$\widehat{\theta}^{FF^*}$&      &   0.032   &       &      &   0.006   &     \\
\hline
\end{tabular}
\end{center}
\end{table}

\begin{table}
\caption{The root mean squared error (rmse) and the absolute mean bias (abias) obtained for models Markov chain (MC) and GARCH(1,1), by considering the empirical quantiles $0.95$, $0.975$ and $0.99$, respectively, $q_{0.95}$, $q_{0.975}$ and $q_{0.99}$.
The direct runs estimator $\widetilde{\theta}^{R}$ is based on run $r=4$ and $r=5$ for, respectively, the MC and GARCH models.
\label{tab2}}
\begin{center}
\begin{tabular}{|l|ccc|ccc|}
\hline
&&rmse&&&abias&\\
\hline
MC  &   $q_{0.95}$ &   $q_{0.975}$    &   $q_{0.99}$ &   $q_{0.95}$ &   $q_{0.975}$    &   $q_{0.99}$\\
\hline
$\widetilde{\theta}^{R}$&   0.084   &   0.122   &   0.202   &   0.024   &   0.053    &   0.115   \\
$\widetilde{\theta}^{I}$&   0.141   &   0.184   &   0.305   &   0.082   &   0.100 &   0.202   \\
$\widehat{\theta}$&   0.071   &   0.089   &   0.141   &   0.036  &   0.013  &   0.032   \\
$\widehat{\theta}^{U}$&   0.071   &   0.089   &   0.141   &   0.022   &   0.005   &   0.063   \\
$\widehat{\theta}^{I}$&   0.118   &   0.148   &   0.226   &   0.032   &   0.063   &   0.136   \\
$\widehat{\theta}^{ML}$&   0.084   &   0.110    &   0.167   &   0.021  &   0.001   &   0.063   \\
$\widehat{\theta}^{SS}$&   0.077   &   0.110    &   0.187   &   0.014  &   0.017   &   0.099   \\
$\widehat{\theta}^{FF}$&   0.071   &   0.077    &   0.105   &   0.053  &   0.060   &   0.078   \\
\hline
$\widehat{\theta}^{FF^*}$&      &   0.055    &      &     &   0.050   &     \\
\hline
\hline
GARCH(1,1)
  &   $q_{0.95}$ &   $q_{0.975}$    &   $q_{0.99}$ &   $q_{0.95}$ &   $q_{0.975}$    &   $q_{0.99}$\\
\hline

$\widetilde{\theta}^{R}$&   0.148   &   0.212   &   0.295   &   0.121   &   0.175    &   0.245   \\
$\widetilde{\theta}^{I}$&   0.200 &   0.221   &   0.315   &   0.130    &   0.117   &   0.215   \\
$\widehat{\theta}$&   0.110   &   0.110 &   0.141   &   0.095   &   0.075   &   0.051   \\
$\widehat{\theta}^{U}$&   0.105   &   0.114   &   0.152   &   0.076   &   0.057   &   0.029   \\
$\widehat{\theta}^{I}$&   0.134   &   0.126   &   0.167   &   0.085    &   0.059   &   0.003   \\
$\widehat{\theta}^{ML}$&   0.010   &   0.110   &   0.152   &   0.073   &   0.057   &   0.027   \\
$\widehat{\theta}^{SS}$&   0.010   &   0.110    &   0.148   &   0.071   &   0.053    &   0.002   \\
$\widehat{\theta}^{FF}$&   0.063   &   0.084    &   0.134   &   0.020   &   0.019    &   0.017   \\
\hline
$\widehat{\theta}^{FF^*}$&      &   0.045    &     &      &   0.010    &      \\
\hline
\end{tabular}
\end{center}
\end{table}

\subsection{Application to financial data}

Log-returns of a financial time series usually present high volatility and clustering of large values.
Klar \emph{et al.} (\cite{klar+12}, 2012) have analyzed DAX German stock market index time series and concluded that GARCH(1,1) is a good model to describe these data. In particular they considered the series of log-returns of DAX closing prices from 1991 to 1998 (see Figure \ref{figDAX}) and fitted a GARCH(1,1) model with autoregressive parameter $\lambda\simeq 0.08$, variance parameter $\beta=0.87$ and innovations $t_7$ (after removing null log-returns). By the tabulated values of the extremal index of GARCH(1,1) models in Laurini and Tawn (\cite{lau+tawn12}, 2012), the true value is around $0.3$. In Table \ref{tabaplic} we report the estimates, derived according to the conclusions of the simulations concerning the GARCH(1,1) model (see also the anti-D$^{(k)}$($u_n$) plots in Figures \ref{figaplicDks} and \ref{figaplicD2Z}). Thus the direct runs estimator $\widetilde{\theta}^{R}$ was computed with run $5$ and the indirect estimators ($\widehat{\theta}$, $\widehat{\theta}^{U}$, $\widehat{\theta}^{I}$, $\widehat{\theta}^{ML}$, $\widehat{\theta}^{SS}$,  $\widehat{\theta}^{FF}$ and $\widehat{\theta}^{FF^*}$) were calculated by considering cycles \z\esp  with $k=5$. The closest values of $0.3$ were obtained with quantile $0.95$ in all cases, which is also in accordance with the simulation study. The indirect upcrossings estimator $\widehat{\theta}^{U}$ presents the nearest approximation, followed by the indirect intervals estimator $\widehat{\theta}^{I}$ and $\widehat{\theta}$. We have also tried other values for $k$ and found that, in this series, $k=6$ leads to the best approximations of $0.3$, with $\widehat{\theta}=0.34$ and the remaining approximately $0.39$, except for the intervals and the direct runs estimator where the estimates were $0.12$ and $0.68$. If we consider the direct runs estimator $\widetilde{\theta}^{R}$ with run $10$ (see Section \ref{sestim}) we obtain the estimate $0.48$.\\

\begin{figure}
\includegraphics[width=5.9cm,height=5.9cm]{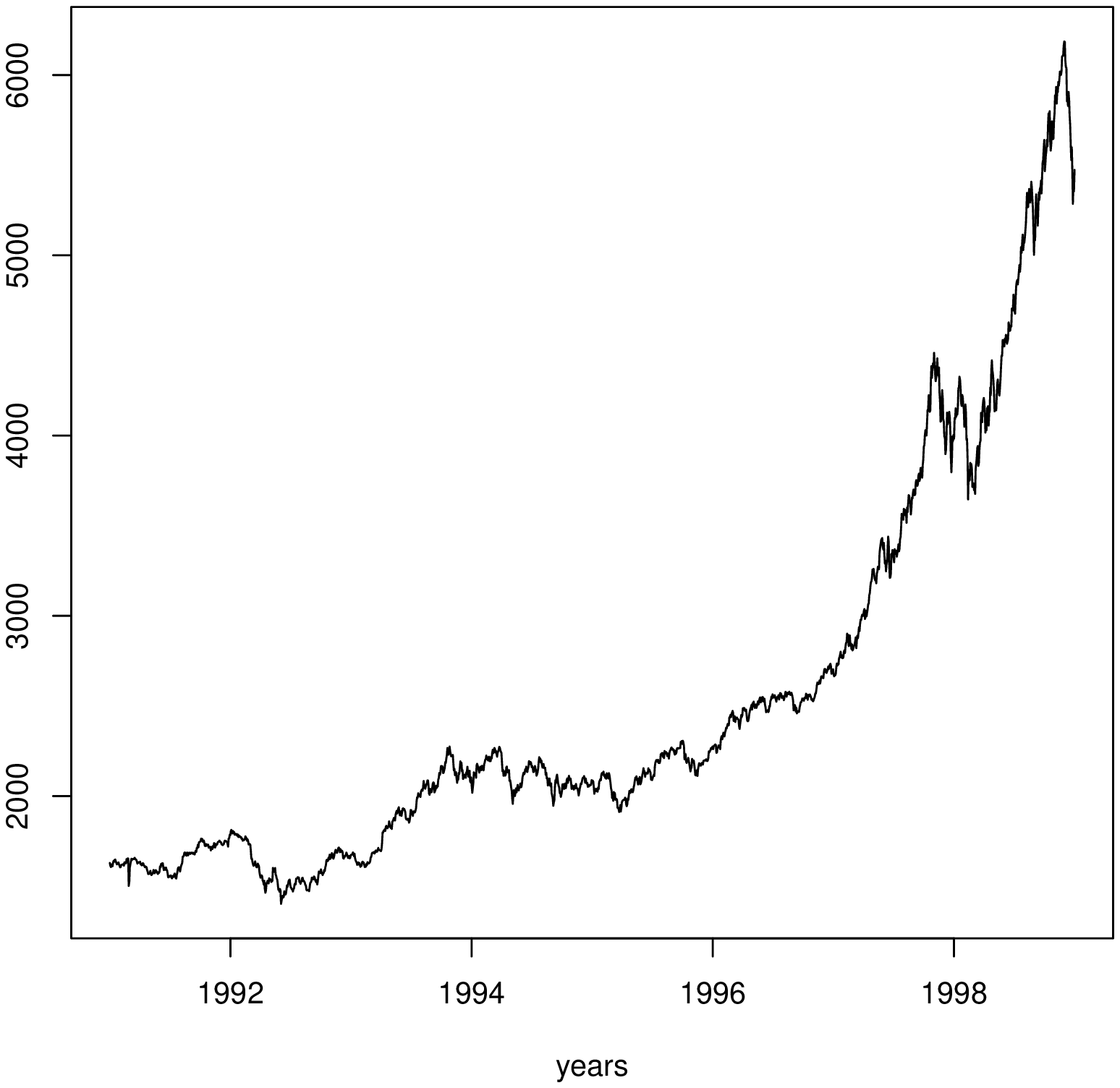}
\includegraphics[width=5.9cm,height=5.9cm]{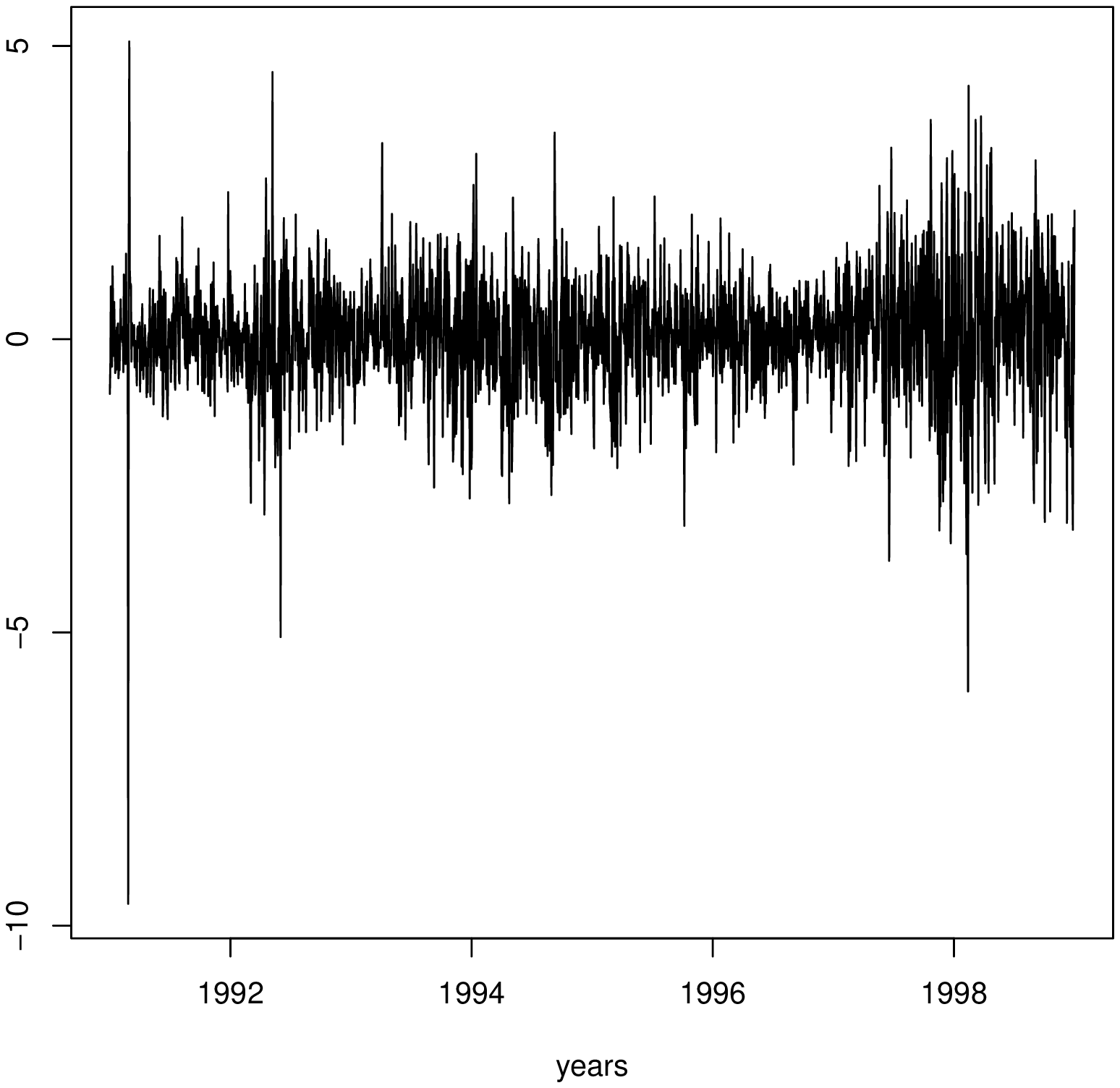}
\caption{Daily closing indexes (left) and daily log-returns (right) of DAX, from 1991 to 1998, with $1786$ observations (successive equal prices excluded).\label{figDAX}}
\end{figure}

\begin{table}
\caption{Estimates of the extremal index of the DAX series at quantile $0.95$. The direct runs estimator was derived with run $5$. The indirect estimators ($\widehat{\theta}$, $\widehat{\theta}^{U}$, $\widehat{\theta}^{I}$,
$\widehat{\theta}^{ML}$, $\widehat{\theta}^{SS}$, $\widehat{\theta}^{FF}$ and $\widehat{\theta}^{FF^*}$) were obtained based on cycles \z\esp  with $k=5$.  \label{tabaplic}}
\begin{center}
\begin{tabular}{ccccccccc}
$\widetilde{\theta}^{R}$ &$\widetilde{\theta}^{I}$&$\widehat{\theta} $&$\widehat{\theta}^{U}$
&$\widehat{\theta}^{I}$
&$\widehat{\theta}^{ML}$&$\widehat{\theta}^{SS}$&$\widehat{\theta}^{FF}$
&$\widehat{\theta}^{FF^*}$\\
\hline
 0.72& 0.50&0.40& 0.36 & 0.37& 0.48 & 0.50 & 0.47& 0.49
\end{tabular}
\end{center}
\end{table}

\begin{figure}
\begin{center}
\includegraphics[width=3.9cm,height=3.9cm]{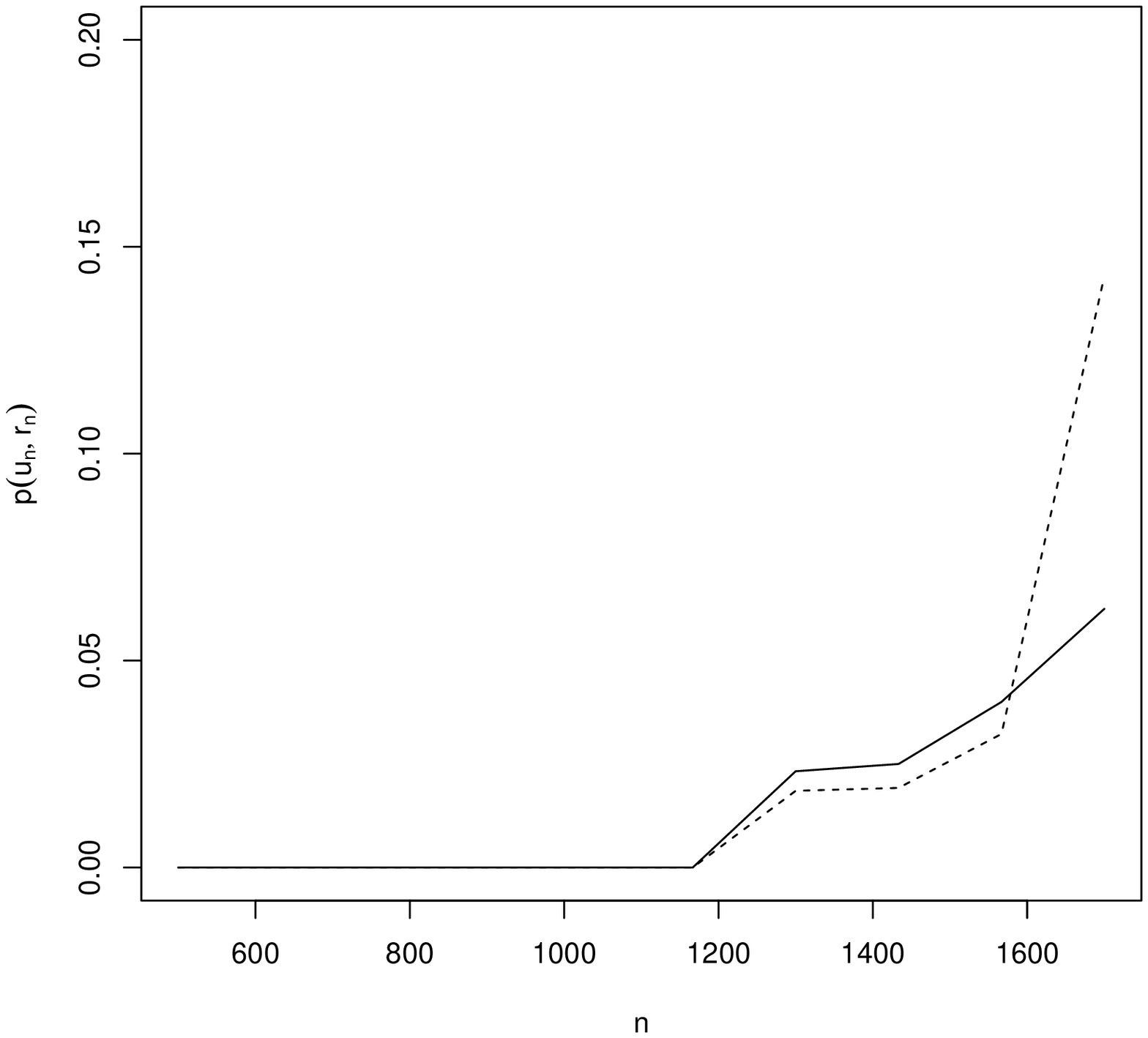}
\includegraphics[width=3.9cm,height=3.9cm]{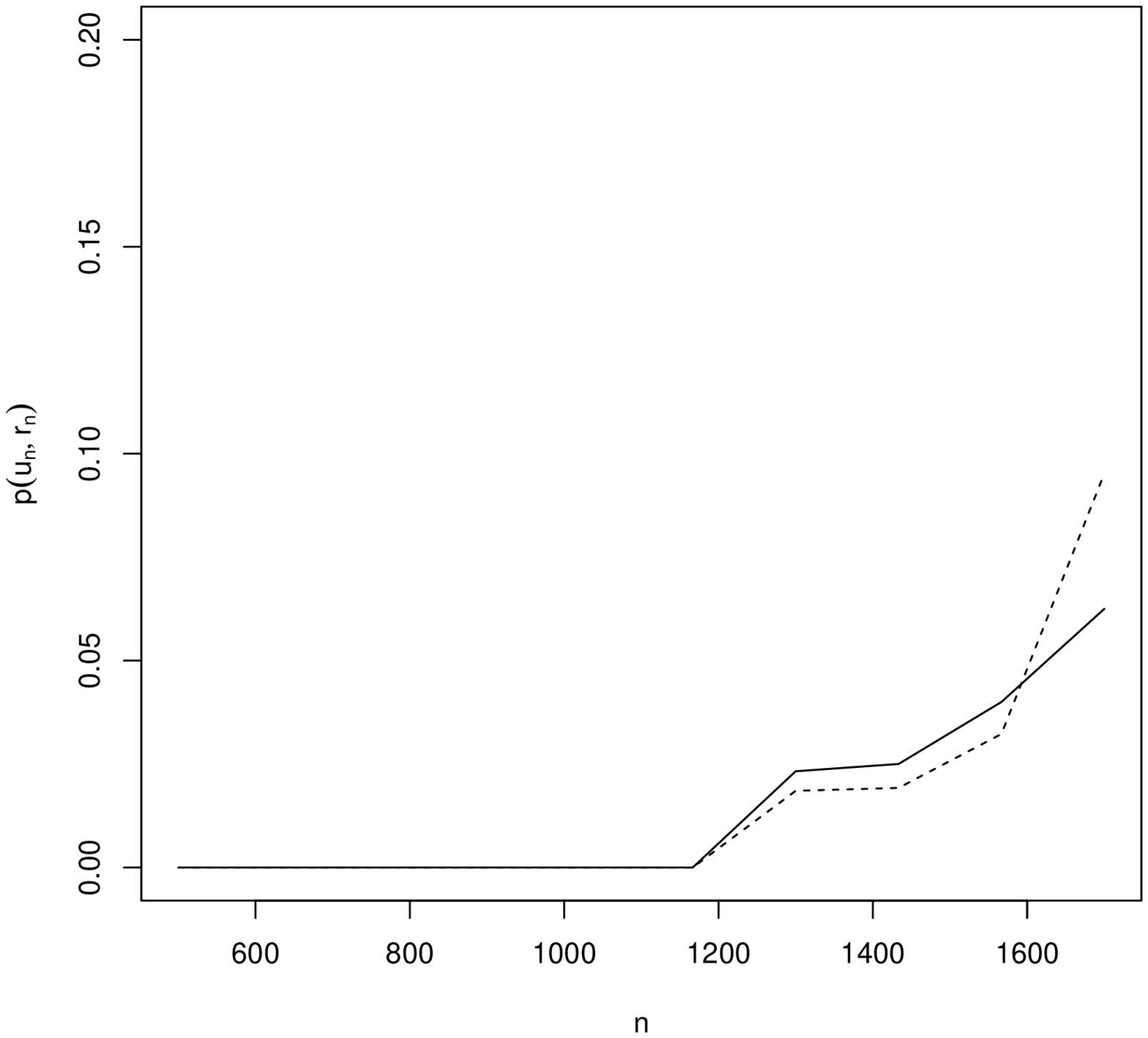}\\
\includegraphics[width=3.9cm,height=3.9cm]{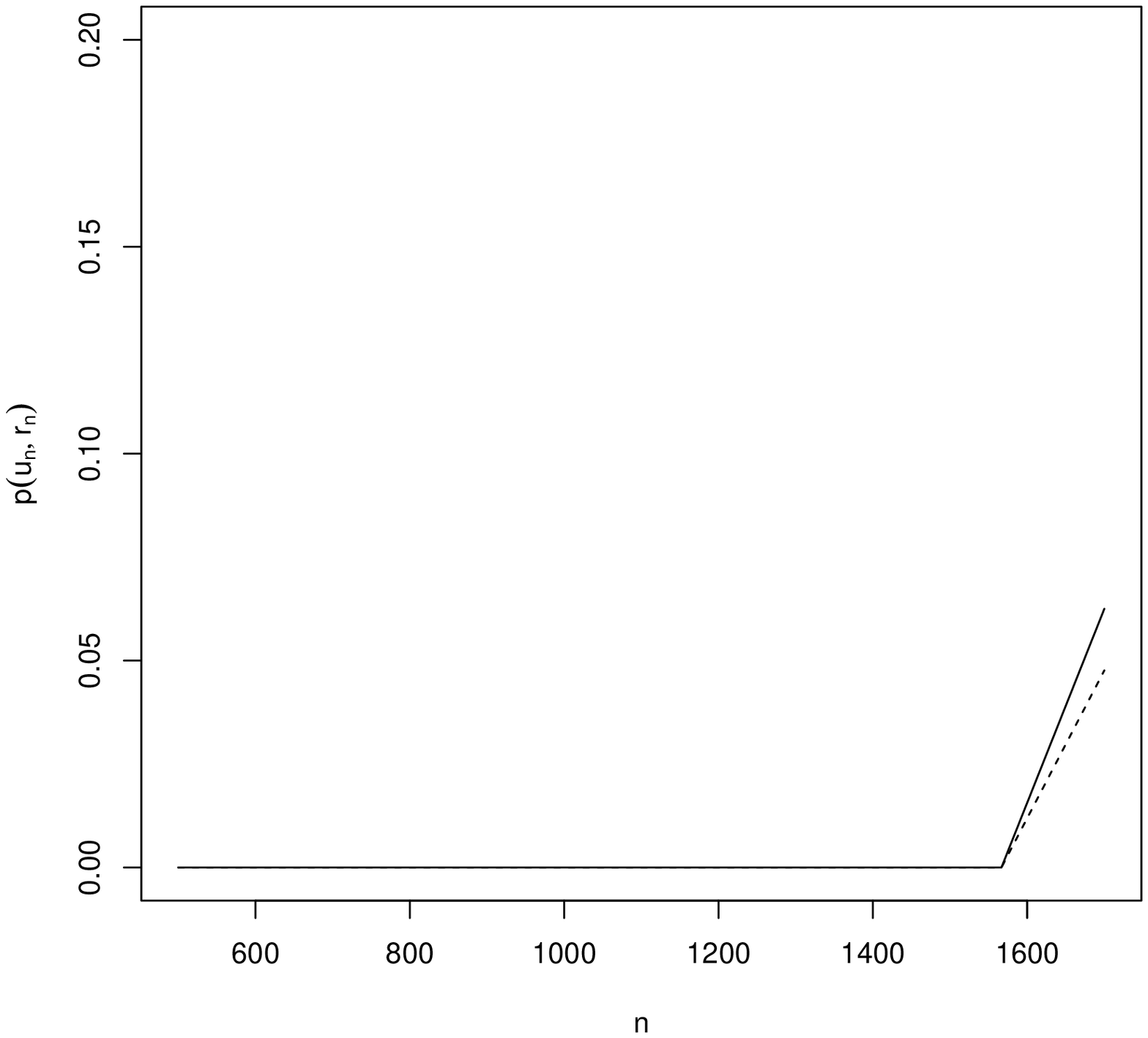}
\includegraphics[width=3.9cm,height=3.9cm]{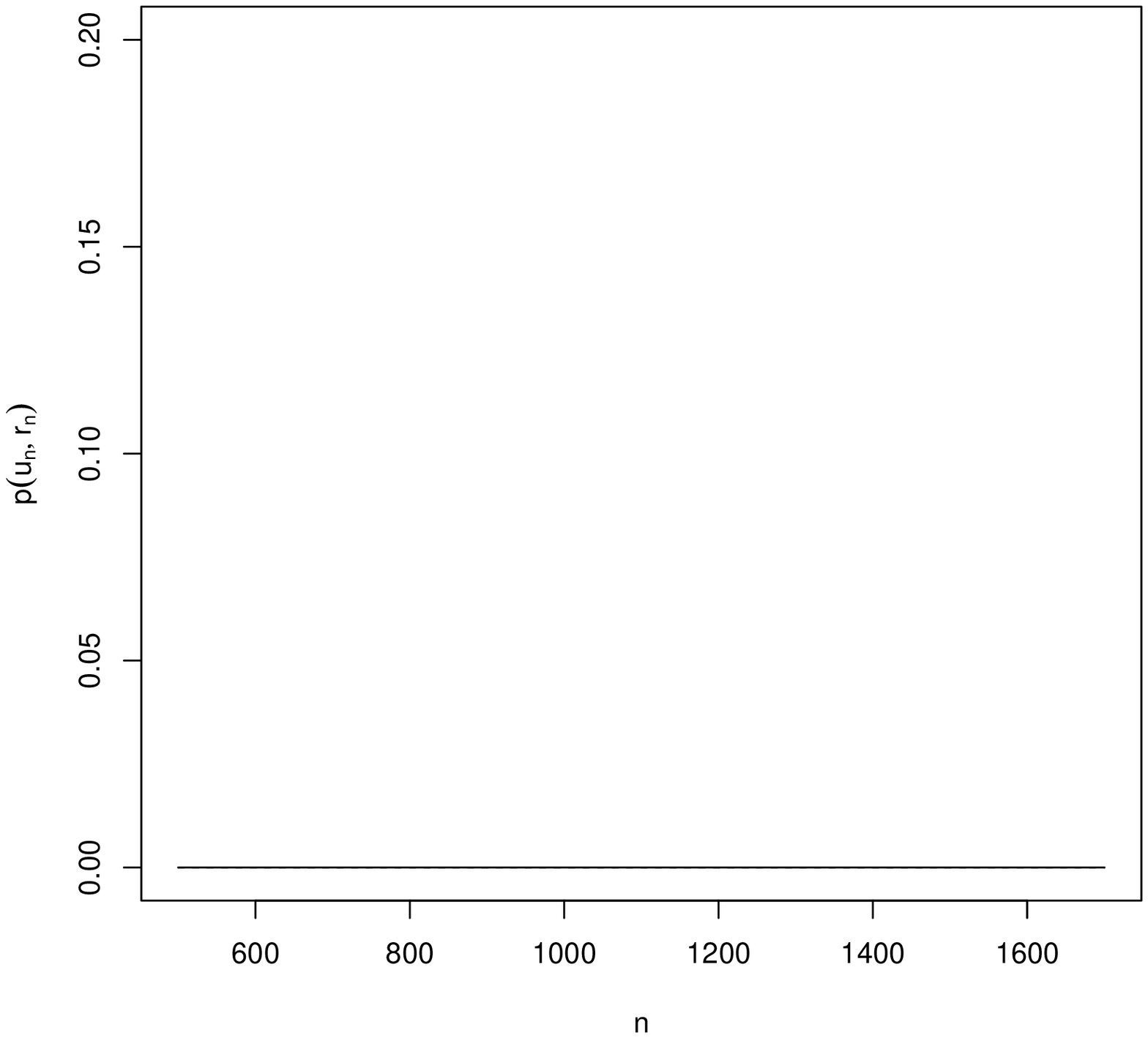}
\caption{From left to right, proportions of anti-D$^{(3)}$($u_n$) to anti-D$^{(6)}$($u_n$) of the DAX series, with $k_n=[(\log n)^{2.5}]$, for $\tau=15$ (full line) and $\tau=20$ (dotted line).\label{figaplicDks}}
\end{center}
\end{figure}

\begin{figure}
\begin{center}
\includegraphics[width=3.9cm,height=3.9cm]{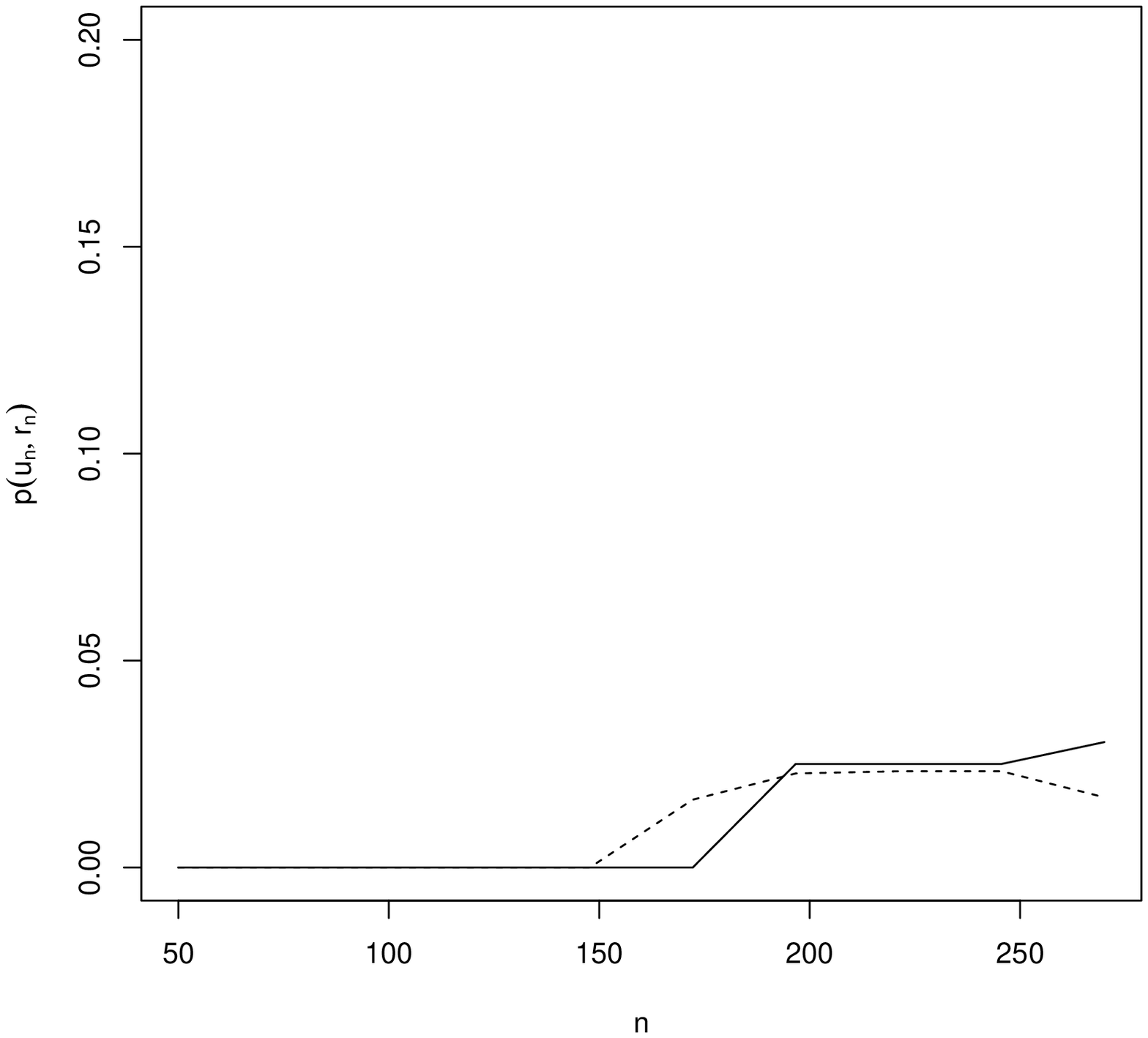}
\includegraphics[width=3.9cm,height=3.9cm]{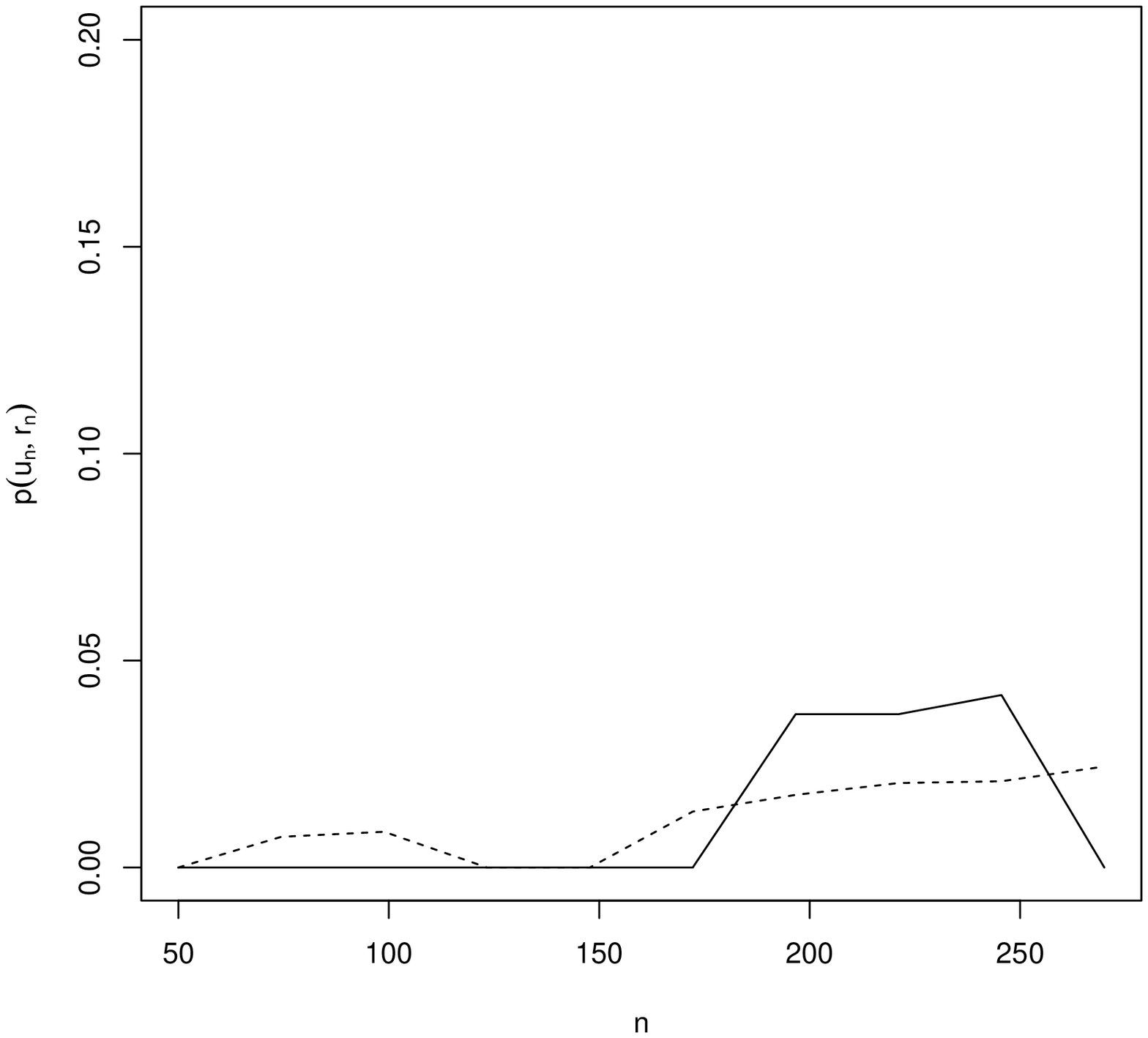}\\
\includegraphics[width=3.9cm,height=3.9cm]{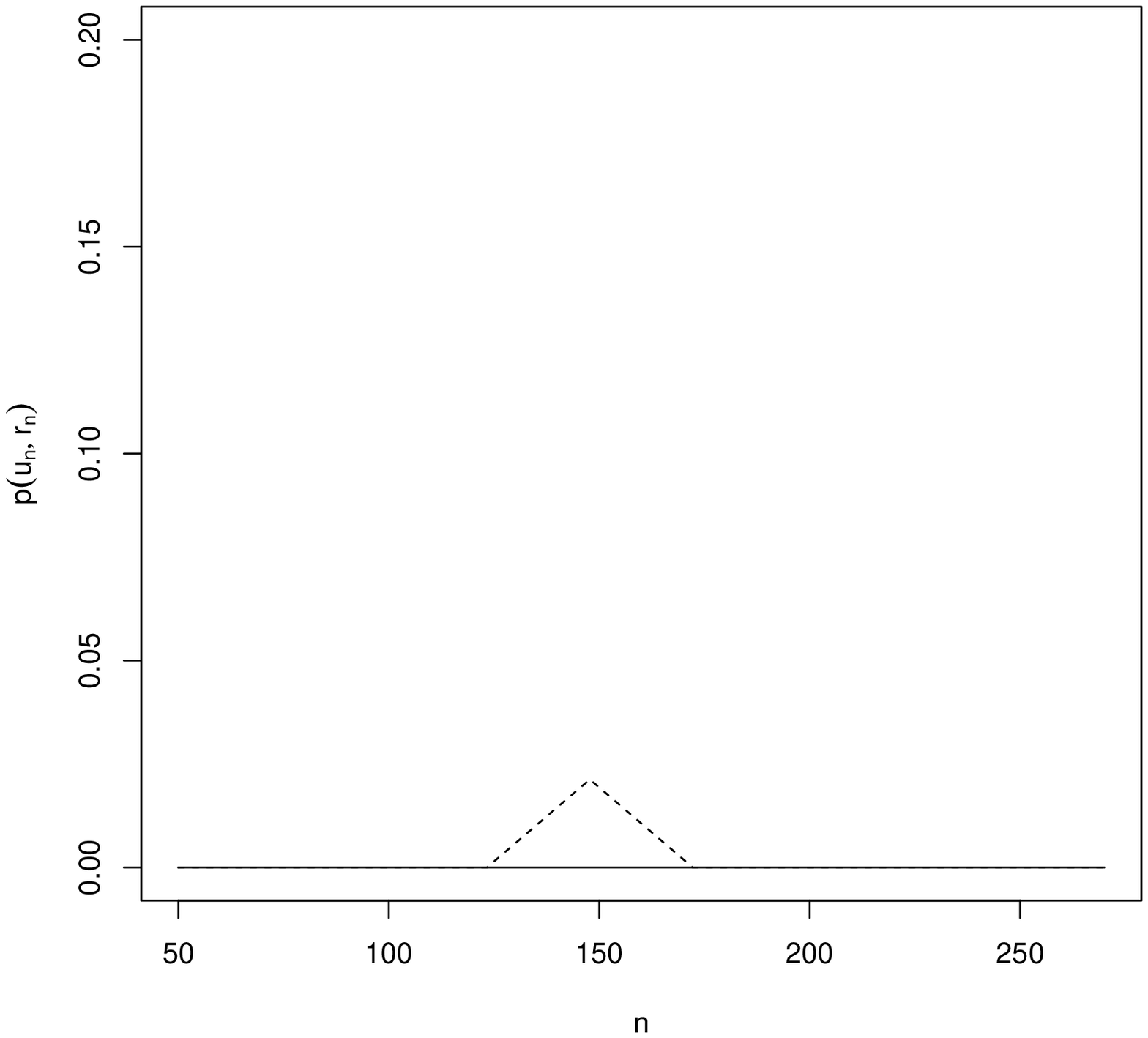}
\includegraphics[width=3.9cm,height=3.9cm]{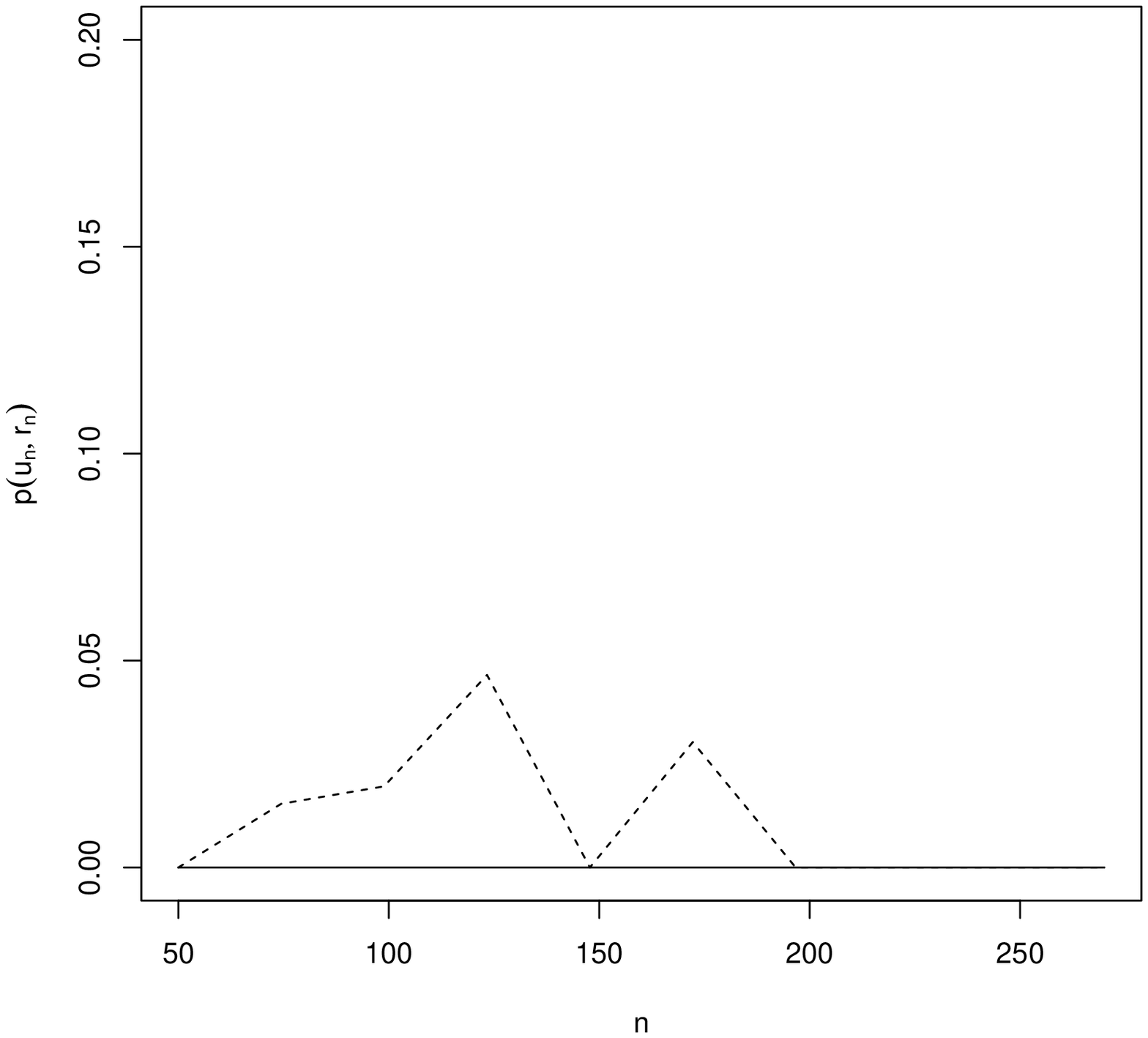}
\caption{From left to right, proportions of anti-D$^{(2)}$($u_n$) of cycles \z\esp of the DAX series with $k=3$ to $k=6$, for $\tau=5$ (full line) and $\tau=10$ (dotted line), with $k_n=[(\log n)^2]$.\label{figaplicD2Z}}
\end{center}
\end{figure}

\section{Conclusions}\label{sdisc}
In this work we consider the estimation of the extremal index, an important dependence parameter within extreme values of stationary sequences. The new approach requires the validity of the local dependence condition D$^{(k)}$($u_n$) of Chernick \emph{et al.} (\cite{chern+91}, 1991). The results are promising under a suitable choice for $k$ and an empirical procedure was proposed for this evaluation. We also find that it is a useful tool for the well-known runs estimator, by guiding a first choice for the run.

\end{document}